\documentclass{amsart}
\usepackage[utf8]{inputenc}

\usepackage{amsmath}
\usepackage{amssymb}
\usepackage{mathrsfs}
\usepackage{enumitem, hyperref}
\usepackage{changepage}
\usepackage{graphics}
\usepackage{pifont}
\usepackage{adjustbox}

\usepackage{tikz}
\usepackage{tikz-cd}
\usepackage{float}
\tikzcdset{scale cd/.style={every label/.append style={scale=#1},
    cells={nodes={scale=#1}}}}

\newtheorem{thm}{Theorem}[section]
\newtheorem{lem}[thm]{Lemma}

\newtheorem{prop}[thm]{Proposition}
\newtheorem{cor}[thm]{Corollary}

\theoremstyle{definition}
\newtheorem{defi}[thm]{Definition}

\newtheorem{ex}[thm]{Example}

\theoremstyle{remark}
\newtheorem{remark}[thm]{Remark}
\newtheorem{question}[thm]{Question}

\newtheorem{warning}[thm]{Warning}

\newcommand{\QM}{\mathbb{Q}}

\DeclareMathOperator{\Hom}{\mathcal{H}om}

\DeclareMathOperator{\End}{End}
\DeclareMathOperator{\Id}{\text{Id}}

\newcommand{\perf}{\mathrm{-Perf}}

\title{Geometric Extensions}
\author{Chris Hone, Geordie Williamson}
\date{June 2023}

\begin{document}

\maketitle

\begin{abstract}
  We prove that the derived direct image of the constant sheaf with
  field coefficients under
  any proper map with smooth source contains a canonical summand. This summand, which we
  call the geometric extension, only depends on the generic fibre. For
  resolutions we get a canonical extension 
  of the constant sheaf. When our coefficients are of characteristic
  zero, this summand is the intersection cohomology sheaf. When our
  coefficients are finite we obtain a new object, which provides
  interesting topological invariants of singularities and topological
  obstructions to the existence of morphisms. The geometric
  extension is a generalization of a parity sheaf.
  Our proof is
  formal, and also works with coefficients in modules over suitably
  finite ring spectra.
\end{abstract}

\section{Introduction}

This paper introduces \emph{geometric extensions}, which are
generalizations of intersection cohomology sheaves and parity
sheaves. We work in the setting of constructible sheaves on algebraic varieties, and show
that direct image sheaves along any resolution contain a canonical
direct summand which is independent of the resolution.\footnote{After
  having discovered this statement and its proof, we became aware of
  the McNamara's paper \cite[\S 5]{PM} where a statement equivalent to
  one of our
  main theorems is proved. Our proof is almost identical to
  McNamara's.} 
When our coefficients are $\QM$, this summand is the intersection
cohomology sheaf. When our coefficients are finite, we obtain a new object. Our proof is formal, and works more generally for proper maps with smooth source, and with coefficients in any suitably finite ring spectrum. The stalks of the geometric extension (with coefficients in finite fields and other ring spectra) provide subtle topological invariants of the singularities of algebraic varieties.

In order to motivate geometric extensions, we first recall the
traditional route to intersection cohomology extensions through
perverse sheaves. We then turn to an alternative approach via the
Decomposition Theorem, which will motivate the consideration of
geometric extensions. We then state our main result, and finally give
some motivation from modular representation theory, where geometric
extensions generalise the notion of a parity sheaf.

\subsection{Motivation from the Decomposition Theorem} \label{sec:motivation}
Let $Y$ be a complex algebraic variety, equipped with its classical (metric) topology. Inside the constructible derived category of sheaves of $\QM$-vector spaces on $Y$ there is a remarkable abelian category of \emph{perverse sheaves}, which is preserved by Verdier duality.  The abelian category of perverse sheaves is finite length and its simple objects are the intersection cohomology extensions of simple local systems on irreducible, smooth, locally closed subvarieties. Their global sections compute intersection cohomology. 

The central importance of intersection cohomology extensions becomes manifest in the Decomposition Theorem. The Decomposition Theorem states that for smooth $X$ and any proper morphism
\[
f : X \to Y
\]
of complex algebraic varieties the derived direct image $f_* \QM_X$ is
\emph{semi-simple}: isomorphic in the derived category to a direct sum
of shifts of intersection cohomology extensions of simple local systems on strata. The Decomposition Theorem implies that the intersection cohomology of $Y$ is a direct summand in the cohomology of any resolution. It is also implies (and generalizes) fundamental ideas in the topology of complex algebraic varieties like the local and global invariant cycle theorems, semi-simplicity of monodromy and Hodge theory \cite{BBD, dCM, Sa, W}.

The Decomposition Theorem also provides another route to intersection cohomology complexes. The constructible derived category is Krull-Schmidt: every object admits a decomposition into indecomposable summands, and this decomposition is unique. The Decomposition Theorem implies that if one considers all proper maps to $Y$ with smooth source
\[
  \begin{tikzpicture}[scale=.7]
    \node (x1) at (-2,2) {$X_1$};
    \node (x2) at (0,3) {$X_2$};
    \node (x3) at (2,2) {$X_3$};
    \node (y) at (0,0) {$Y$};
    \draw[->] (x1) to node[above] {\small $f_1$} (y);
    \draw[->] (x2) to node[right] {\small $f_2$} (y);
    \draw[->] (x3) to node[below,right] {\small $f_3$} (y);
  \end{tikzpicture}
  \]
  then the summands of the derived direct images $(f_i)_{*} \QM_{X_i}$ are of a special form: they are shifts of intersection cohomology complexes.

  This observation allows one to imagine an alternate version of history, where intersection cohomology complexes were discovered via the Krull-Schmidt theorem rather than through the theory of perverse sheaves\footnote{The expert will miss the adjective ``of geometric origin'' in this discussion. Everything we discuss in this paper will be ``of geometric origin''.}. It also naturally raises the following question:

  \begin{question} \label{q:summands}
    Let $\Lambda$ denote a ring, and let
    $\Lambda_X$ denote the constant sheaf on $X$ with coefficients in
    $\Lambda$. What can one say about the   summands of the derived
    direct image $f_* \Lambda_X$, for any resolution $f : X \to Y $?
    More generally, what can one say about the summands of $f_*
    \Lambda_X$ for any proper morphism with smooth source $f : X \to Y$?
  \end{question}

  This question should be considered the central motivation of this paper.

  By the proper base change theorem, the stalks of $f_* \Lambda_X$
  record the $\Lambda$-cohomology of the fibres of $f$. If Question
  \ref{q:summands} has an answer giving a small list of possible
  summands (as is the case with $\Lambda = \QM$, as implied by the
  Decomposition Theorem) then there are basic building blocks of the
  cohomology of morphisms, which only depend on $Y$ and not on the
  particular morphism. For example, the ``support theorems''
  (see e.g. \cite{BorhoMacPherson,ngo, MV}) show that the fibres of certain ``minimal''
  maps (e.g. the Grothendieck-Springer resolution or Hitchin fibration) are determined to a large extent
  by the base $Y$, and the generic behaviour of the map.

  \subsection{Main results}   Let $\Lambda$ denote a field or complete
  local ring. As above, $\Lambda_X$ denotes the constant sheaf on $X$
  with coefficients in $\Lambda$. We are able to give a partial answer to Question \ref{q:summands}. 
  Any resolution contains a canonical direct summand:

  \begin{thm}(see Theorem \ref{thm:main1}) \label{thm:main1-intro} Let $Y$ be an
    irreducible variety.
    There exists a complex $\mathscr{E}(Y,\Lambda) \in D^b_c(Y,\Lambda)$ characterised up to isomorphism by the following:
    \begin{enumerate}
    \item $\mathscr{E}(Y,\Lambda)$ is indecomposable and its support
      is dense;
    \item $\mathscr{E}(Y,\Lambda)$ is
    a summand inside $f_* \Lambda_X$, for any resolution $f: X \to Y$.
    \end{enumerate}
  \end{thm}

  We call $\mathscr{E}(Y,\Lambda)$ the \textbf{geometric extension} on
  $Y$.

    \begin{remark}
    When $\Lambda = \QM$ then $\mathscr{E}(Y,\Lambda) = IC(Y,\QM)$,
    by the Decomposition Theorem (see Proposition \ref{ICprop}).
  \end{remark}

  The stalks of the geometric extension record behaviour which ``has
  to be there in any resolution''. Indeed, the proper base change
  theorem and Theorem \ref{thm:main1} immediately imply:

  \begin{cor} \label{cor:forced}
    Suppose $\Lambda$ is a field.
      For any resolution $f : X \to Y$ one has
\[
      \dim H^i(\mathscr{E}(Y,\Lambda)_y) \le \dim
      H^i(f^{-1}(y), \Lambda)\]
for any $y \in Y$.
\end{cor}

\begin{remark} \label{rem:nosemismall}
  This corollary can be used to rule out the existence of resolutions
  of a particular form. For example, if $\mathscr{E}(Y,\Lambda)_y$ has
  non-zero stalks in degrees $0$ and $2m$ for some $m$, then Corollary
  \ref{cor:forced} and the existence of fundamental classes implies that any
  resolution of $Y$ has to have fibres dimension at least $m$ over $y$
  (see Example \ref{ex:A/pm}). This can be used to prove the
  non-existence of semi-small
  resolutions: if $\mathscr{E}(Y,\Lambda)$ is not perverse, then no
  semi-small resolution of $Y$ exists (see Proposition \ref{prop:no-semi-small}).
\end{remark}

  \begin{remark}
    Theorem \ref{thm:main1} has also been obtained by McNamara \cite[\S 5]{PM},
    with a very similar proof. McNamara also noticed
    Remark \ref{rem:nosemismall} and uses this observation to rule out the
    existence of semi-small resolutions of certain Schubert varieties.
  \end{remark}

In the setting of Decomposition Theorem, it is
essential to take local systems into account. This is already the case
for a smooth morphism with smooth target $f : E \to X$, where the Decomposition Theorem
implies that $f_*\QM_E$ splits as a direct sum of its cohomology
sheaves $\mathcal{H}^i(f_*\QM_E)$, and each of the resulting local systems (with stalks
$H^i(f^{-1}(x),\QM)$) are semi-simple.

When we take more general coefficients, it is no longer true that the
direct image along a proper smooth morphism has to split, nor that the
resulting local systems are semi-simple. It is easy to produce
examples where the monodromy fails to be semi-simple when the
coefficients are not of characteristic $0$. The failure of the direct
image to split in the derived category is a little more
subtle. We give
examples of this failure to split for $\mathbb{P}^1$-bundles (where
the non-splitting is connected to the Brauer group) in Example
\ref{ex:nonss}.

This motivates us to consider geometric local systems. A
\textbf{geometric local system} $\mathscr{L}$ on $U$
is a smooth and proper map with smooth target\[V\xrightarrow{\mathscr{L}}U.\]
The following generalises Theorem \ref{thm:main1-intro} to take local systems into account:

  \begin{thm}(see Theorem \ref{thm:main2}) 
 \label{thm:main2-intro} Assume $Y$ is irreducible. For any dense (smooth) $U \subset Y$ and geometric local system $V
    \xrightarrow{\mathscr{L}}U$ there is a unique complex
    $\mathscr{E}(Y,\mathscr{L}) \in D^b_c(Y,\Lambda)$  satisfying:
    \begin{enumerate}
    \item $j^*\mathscr{E}(Y, \mathscr{L}) \cong
      \mathscr{L}_*\Lambda_U$ where $j : U \hookrightarrow Y$ denotes
      the inclusion;
      \item      $\mathscr{E}(Y, \mathscr{L})$ has no summands supported on the
      complement of $U$;
    \item for any proper map with smooth source $f : X \to Y$ which agrees with
      $\mathscr{L}$ over $U$, $\mathscr{E}(Y,\Lambda)$ occurs as a
      summand of $f_* \Lambda_X$. 
    \end{enumerate}
\end{thm}

We call $\mathscr{E}(Y, \mathscr{L})$ the \textbf{geometric extension}
of the geometric local system $\mathscr{L}$.

\begin{remark}\label{ICdef} We explain what Theorem \ref{thm:main2} says when
  $\Lambda = \QM$. By the (smooth case of the) Decomposition Theorem,
  $\mathscr{L}_*\QM_V$ is isomorphic to the direct sum of its intersection
  cohomology sheaves, $\bigoplus \mathscr{H}^i
  (\mathscr{L}_*\QM_V)[-i]$, and each cohomology sheaf $\mathscr{H}^i
  (\mathscr{L}_*\QM_V)$ is semi-simple. If we define
  $IC(\mathscr{L}_*\QM_V)$ to be $\bigoplus IC(Y, \mathscr{H}^i
  (\mathscr{L}_*\QM_V))[-i]$,  then $\mathscr{E}(Y,\mathscr{L}) = IC(Y, \mathscr{L}_*\QM_V)$,
    by the Decomposition Theorem.
  \end{remark}

  \begin{remark} Again, $\mathscr{E}(Y,\Lambda)$ provides lower bounds
    on the cohomology of any proper morphism extending
    $\mathscr{L}$. We leave it to the reader to formulate an analogue
    of Corollary \ref{cor:forced} in this more general setting.
  \end{remark}

  \begin{warning}
  In contrast to the setting over $\QM$, we prove that 
  $\mathscr{E}(Y,\mathscr{L})$ is not determined by its restriction to
  $U$. More precisely, using the Legendre family of elliptic curves, we produce two geometric local systems $V'
    \xrightarrow{\mathscr{L}'}U$ and $V''
    \xrightarrow{\mathscr{L}''}U$ which have the same monodromy over
    $\mathbb{F}_2$, but whose geometric extensions are not isomorphic (see
    Example \ref{ex:legendrefamily}).\end{warning}

  \subsection{Coefficients in ring spectra} One interesting aspect of
  the current paper is that the results are formal: we only need the
  proper base change theorem, the existence of fundamental classes, and some 
  finiteness to ensure the Krull-Schmidt theorem. (In the body of the
  paper we axiomatise our setup as a \textbf{base change formalism}, and
  prove our results in that setting.)

  Using our formalism, we deduce that our main theorems
  hold with coefficients in suitable stable $\infty$-categories. In
  case the reader (like the authors) is intimidated by this
  theory, we provide a few paragraphs of motivation as to why we are
  interested in this level of generality.\footnote{
  $\infty$-categories are not needed in any arguments
  in this paper. However, $\infty$-categories are needed in providing
  the input (the ``base change formalism'') with which we work.}

    A major theme in
    homotopy theory is the consideration of generalized cohomology
    theories like K-theory, elliptic cohomology, Brown-Peterson
    cohomology and the Morava K-theories. One can think about all
    of these cohomology theories as lenses through which to view
    homotopy theory: facts which are transparent in one theory are often
    opaque in another. Computation plays an enormously important role,
    and computations are often performed using the fact that any map is
    homotopic to a fibration, which gives rise to useful spectral sequences.

    In algebraic geometry, smooth morphisms (the
    algebraic geometer's fibrations) are extremely rare, and an
    important role is played by constructible sheaves and the six
    functor formalism. The spectral sequence of a fibration is
    replaced by the Leray-Serre spectral sequence, or its
    variants. The Decomposition Theorem is a very powerful tool, as it
    allows one to conclude that the perverse Leray-Serre spectral
    sequence degenerates for any proper map.
    Traditionally, this formalism only encompases
    cohomology, homology and its variants. The connection to
    cohomology is via the basic fact that the derived global sections
    of the constant sheaf compute cohomology.

    In homotopy theory, it has been clear for decades that one can obtain generalised
    cohomology as the global sections of a local object. (Indeed, by
    Brown representability, the generalized $E$-cohomogy of $X$ is given
    by homotopy classes of maps $[X, E^i]$, where $E^i$ represents
    $i^{th}$ $E$-cohomology.) Thus it is natural to ask:
    is there some theory of constructible sheaves, which would allow
    one to push and pull constant $E$-sheaves in much the same way
    that one can push and pull constant sheaves in algebraic geometry?
    Such a theory would unify the two approaches to cohomology of the
    proceeding two paragraphs.\footnote{For an excellent articulation
      of this question, see \cite{DanP}.}

    Building on the fundamental work of Lurie
    \cite{LurieHTT,LurieHA,LurieSAG}, such a theory has become
    available \cite{VolpeM}. We
    believe these more general coefficients (e.g. Morava
    K-theories) will provide a powerful tool to study torsion phenomena in the topology of
    complex algebraic varieties, in much the same way as they have done in homotopy
    theory. It is for this reason that we work in the generality of
    sheaves with coefficients in certain $\infty$-categories. (Again,
    we emphasise that we only need very formal properties from this
    theory, and none of its internals.) However,
    we do not discuss any computations with these more general objects
    in this paper. Geometric extensions in greater generality
    play an important role in forthcoming work of the first author
    \cite{Hone-thesis} and Elias and the second author
    \cite{EW-Satake}. The idea of taking summands in more general
    (motivic)     cohomology theories also shows up in the work of Eberhardt
    \cite{Eberhardt1, Eberhardt2}.

    \begin{remark}
      In \S\ref{sec:motivation} we discussed two routes to
      intersection cohomology sheaves: one via the theory of perverse
      sheaves (abelian categories), and one via the Decomposition
      Theorem and Krull-Schmidt (additive categories). In the setting
      of the more exotic 
      coefficients discussed above, one often encounters periodic
      cohomology theories. (This is the case for K-theory,
      as well as all the Morava K-theories.) It is interesting to
      note that (the homotopy category of) sheaves of modules over such
      spectra cannot support a non-trivial $t$-structure, so there is
      no analogue of perverse sheaves with these
      coefficients. Geometric extensions, on the other hand, make
      sense as long as the coefficients satisfy a Krull-Schmidt
      condition.
    \end{remark}
    
    \begin{remark} Above, our discussion centered on constructible
      sheaves (algebraic geometry) and generalized cohohomology
      (homotopy theory). Another major motivation for the development
      of a sheaf theory underlying cohomology theories is the theory
      of triangulated categories of motives (see \cite[Introduction \S A]{CD} for an
      excellent historical introduction).     Our results have a strong motivic flavour, as the reader may have
    already sensed in our definition of a geometric local system. It
    should be emphasised, however, that morphism categories in
    categories of motives rarely have the finiteness conditions that
    we need in this paper.
  \end{remark}

  \begin{remark}Ever since the discovery of intersection cohomology
    in the 1970s, it has been suggested that there should be a
    reasonable theory of intersection K-theory. Such a definition has
    recently been given by P{\u{a}}durariu \cite{IKT}, as a
    subquotient of a geometric filtration on K-theory. The notion of
    geometric extension with coefficients in (rationalised)
    $KU$-modules provides another possible definition of intersection
    K-theory (see Definition \ref{defKtheory}). It would be interesting to
    compare the two approaches.

    One can also hope that
    there is some (abelian, exact, triangulated, stable $\infty$,
    \dots) category $\mathcal{C}$ associated 
    to our space $X$ which categorifies intersection K-theory. The
    current work suggests a possible route towards such a category (at
    least in examples). Namely, intersection K-theory is realised as a
    summand inside the K-theory of any resolution, and the
    isomorphisms for different resolutions are sometimes realised by
    fundamental classes of correspondences. It would be very
    interesting to know if the classes realizing these isomorphisms
    could be lifted to functors, inducing categorical idempotents on
    categories of coherent sheaves on resolutions.
\end{remark}
  
  \subsection{Motivation from Modular Representation Theory}\label{sec:motivmodrep} A major
  motivation for the current work comes from geometric modular representation
  theory. In the work of Lusztig and others, geometric methods
  (e.g. Deligne-Lusztig theory, character sheaves, the Kazhdan-Lusztig
  conjecture) have  played a decisive role in classical
  (i.e. characteristic $0$) representation
  theory. Modular geometric representation theory aims to transport
  these successes to modular (i.e. mod $p$) representation theory (see
  \cite{JMW-survey,Achar,W-ICM}). 

  In this theory the notion of a \textbf{parity sheaf} has come to play a
  central role. These are sheaves whose stalks and costalks
  vanish in either even or odd degrees. In \cite{JMW} it is proved
  that on many varieties arising in geometric representation theory
  parity sheaves are
  classified in the same way as intersection cohomology
  complexes. Their importance in geometric modular representation
  theory appears to stem from two sources:
  \begin{enumerate}
  \item Whilst it is extremely difficult to compute with intersection
    cohomology sheaves with modular coefficients, computations with
    parity sheaves are sometimes possible, thanks to the role of  intersection forms \cite[\S 3]{JMW}. This
    computability is behind counter-examples to the bounds in
    Lusztig's conjecture arising from unexpected torsion
    \cite{W-explosion,Wtors} and the billiards conjecture of Lusztig
    and the second author \cite{LW}.
  \item When establishing derived equivalences, it is often useful to
    have a good class of generators whose algebra of extensions is
    formal. With $\QM$-coefficients, intersection cohomology complexes
    often provide such objects. When working with modular
    coefficients, parity sheaves seem to play the role of ``pure''
    objects, although it is still somewhat mysterious as to why (see
    \cite{RSW,ARKoszul,ARloop,AMRW}).
  \end{enumerate}

  The main theorem of \cite{JMW} relies crucially on 
  the vanishing of odd cohomology of the strata in a fixed
  stratification. These properties often hold in geometric
  representation theory, but can be a hindrance. For example, they
  can be destroyed by passing to a normal slice.

  Geometric extensions address this deficiency: 
  parity sheaves are very often geometric extensions. Consider a
  stratified variety $X = \bigsqcup X_\lambda$ satisfying the conditions
  of \cite[2.1]{JMW}, so that the notion of a parity sheaf makes
  sense. In almost all examples of parity sheaves (for the constant
  pariversity) one has nice\footnote{i.e. \emph{even} in the language
    of \cite[\S 2.4]{JMW}} resolutions
\[ \pi_\lambda : \widetilde{X}_\lambda \to \overline{X_\lambda} \]
such that the parity sheaf corresponding to the stratum $X_\lambda$ is
an indecomposable direct summand of $(\pi_\lambda)_*
\Lambda_{\widetilde{X}_\lambda}$. It follows from Theorem
\ref{thm:main1} that the parity sheaf coincides with the geometric
extension $\mathscr{E}(Y,\Lambda)$.

\begin{remark}
  As we remarked above, Theorems \ref{thm:main1} and \ref{thm:main2}
  provide a partial answer to our guiding Question
  \ref{q:summands}. Namely, indecomposable summands with dense support
  are geometric extensions. However, our theorems say nothing about what happens
  on lower strata. One could hope that they are geometric extensions,
  but we have very limited evidence for this claim. (The issue is that, in contrast to the situation for IC and parity sheaves, we have no characterisation of the geometric extension which
  is intrinsic to the space.) In the setting of parity sheaves (where one does have a
  local characterisation in terms of stalks and costalks) it is true that all
  summands are parity sheaves, which can be considered a weak form of
  the Decomposition Theorem.
\end{remark}

\subsection{Acknowledgements} We would like to thank Bhargav Bhatt for
helpful comments, in particular he suggested Example
\ref{ex:nonss}. We would like to thank Burt Totaro for useful comments
on Example \ref{ex:A/pm}. 
We would like to thank Roman Bezrukavnikov, Peter McNamara, Luca
Migliorini, Marco Volpe and Allen Yuan for useful discussions.

\section{Base change formalism}

In this section we will describe the categorical formalism used to prove Theorem 1. This section is purely $2$-categorical, and we shall proceed axiomatically to emphasise its formal nature. The reader comfortable with the formalism of the constructible derived category of sheaves will find nothing unfamiliar in what follows. 
\\
\\
From here, $\mathscr{C}$ will be a category with pullbacks and a terminal object $\ast$.

\begin{defi}\label{basechangeformalism}
A $\mathbf{base}$ $\mathbf{change}$ $\mathbf{formalism}$
$S:=(S_*,S_!)$ on $\mathscr{C}$ is a data of a pair of pseudo-functors
$S_*,S_!$ from $\mathscr{C}$ to the $2$-category $\mathbf{Cat}$, and a
lax natural transformation of pseudofunctors $c:S_!\rightarrow
S_*$. These two functors $S_*$ and $S_!$ strictly agree on objects,
and the object components of $c$ are the identity functor. For any
morphism $f$ in $\mathscr{C}$ we abbreviate $S_X=S_*(X)=S_!(X)$, $f_*
= S_*(f)$, $f_! = S_!(f)$, and $c_f:f_!\rightarrow f_*$ for the
component of $c$ at a morphism $f$ in $\mathscr{C}$. We require the
following: 
\begin{enumerate}\label{BC}
\item[(BC1)]
 For all morphisms $f$, $f_*$ admits a left adjoint $f^*$, and $f_!$ admits a right adjoint $f^!$.
\end{enumerate}
In view of (BC1), we say a morphism $f$ in $\mathscr{C}$ is $\mathbf{proper}$ if $c_f:f_!\rightarrow f_*$ is an isomorphism, and $\mathbf{\acute{e}tale}$ if there exists an isomorphism $f^*\cong f^!$.

For the remaining two conditions, we fix a pullback square:
\begin{equation}\tag{PS}
    \begin{tikzcd}
        X'\arrow[d,"\tilde{f}"]\arrow[r,"\tilde{g}"]&X\arrow[d,"f"]\\
        Y'\arrow[r,"g"]& Y
    \end{tikzcd}
\end{equation}
Our final conditions are the following:
\begin{enumerate}
\item[(BC2)]\label{BC2}
In (PS), if $f$ is étale (resp. proper), then $\tilde{f}$ is also étale (resp. proper).
\item[(BC3)] \label{BC3}
In (PS), the induced base change morphisms
\end{enumerate}
\[g^*f_*\rightarrow \tilde{f}_*\tilde{g}^*\]
\[\tilde{f}_!\tilde{g}^!\rightarrow g^!f_!\]
are both isomorphisms if $f$ is proper, or if $g$ is étale (see Remark \ref{basechangedef} for the definition of these base change morphisms).
\end{defi}

\begin{remark}
    In many settings (constructible sheaves on complex varieties,
    étale sheaves, $D$-modules,…) one encounters a ``6-functor formalism". Usually this manifests as a collection of
    triangulated (or $\infty-$) categories, with six functors \[f_*,
      f_!, f^!, f^!, \Hom, \otimes\] satisfying a raft of relations
    (see e.g. \cite{dCM}). There has been recent progress on
    axiomatizing what a six functor formalism is, particularly in the
    setting of $\infty-$categories (e.g. \cite{CD, mann, liuzheng,
      GR1}). As far as we are aware this process is ongoing
    and there is still no generally accepted definition. We need very
    little from the theory, and have tried to isolate the key features
    we require in Definition \ref{basechangeformalism}. The reader
    should have little trouble adapting other settings (e.g. étale
    sheaves, or $D$-modules) to our axioms. 
\end{remark}

The following will be a recurring example throughout this paper.
\begin{ex}\label{runningex}
  Let $\mathscr{C}$ be the category of complex algebraic varieties, and let $S_X=D^b_c(X,k)$ be the constructible derived category of sheaves on $X(\mathbb{C})$ with coefficients in a field $k$. (It is important that the derived category is used here, since $f^!$ does not exist in general as a functor on abelian categories.) 
  In this framework, the notions of étale and proper match their
  topological definitions, hence are closed under pullbacks giving
  (BC2). Our third axiom (BC3) goes under the name of proper base
  change in the literature (e.g. \cite[Proposition 2.5.11]{kashiwara2002sheaves}.)
\end{ex}

\begin{remark}\label{basechangedef}
Explicitly, the data of a base change formalism is an assignment of a category $S_X$ to each object $X$ in $\mathscr{C}$, functors $f_!:S_X\rightarrow S_Y$ and $f_*:S_X\rightarrow S_Y$ for each morphism $f:X\rightarrow Y$, and coherent isomorphisms $f_*\circ g_* \cong (f\circ g)_*$, $f_!\circ g_!\cong (f\circ g)_!$, along with natural transformations $c_f:f_!\rightarrow f_*$, satisfying some compatibilities. We will suppress these compatibility $2$-isomorphisms for $f_*$ and $f_!$, but the reader should bear in mind that they are a critical part of our input data, as they supply the middle maps used for the base change morphisms of (BC3):\[g^*f_*\xrightarrow{\eta} \tilde{f}_*\tilde{f}^*g^*f_*\cong \tilde{f}_*\tilde{g}^*f^*f_*\xrightarrow{\epsilon} \tilde{f}_*\tilde{g}^*\] \[\tilde{f}_!\tilde{g}^!\xrightarrow{\eta} \tilde{f}_!\tilde{g}^!f^!f_!\cong \tilde{f}_!\tilde{f}^!g^!f_!\xrightarrow{\epsilon} g^!f_!\]
\end{remark}
\subsection{The convolution isomorphism.}
\begin{defi}\label{conviso} Consider a pull-back square, with $g$ proper:
\[\begin{tikzcd}
    X\times_Y X'\arrow[d,"\tilde{g}"]\arrow[r,"\tilde{f}"]&X'\arrow[d,"g"]\\
    X\arrow[r,"f"]& Y
\end{tikzcd}\]
We have the following natural isomorphism of functors, which we call the \textbf{convolution isomorphism}:
\[\begin{tikzcd}
\Hom(f_!\_,g_*\_)\arrow[rr,"\tau_{f,g}","\sim"']&&\Hom(\tilde{g}^*\_,\tilde{f}^!\_)
\end{tikzcd}\]

This is defined as the composition of the following isomorphisms:
\begin{align*}
\Hom(f_!\_,g_*\_)\rightarrow \Hom(\_,f^!g_*\_)  \rightarrow
  \Hom(\_,f^!g_!\_) & \rightarrow \\ \rightarrow
  \Hom(\_,\tilde{g}_!\tilde{f}^!\_) & \rightarrow \Hom(\_,\tilde{g}_*\tilde{f}^!\_)\rightarrow \Hom(\tilde{g}^*\_,\tilde{f}^!\_).
\end{align*}
By evaluating this isomorphism on objects $F$ in $S_X$ and $G$ in $S_{X'}$, we obtain the pointwise convolution isomorphism:\[\Hom(f_! F,g_* G)\xrightarrow{\tau_{f,g}}\Hom(\tilde{g}^*F,\tilde{f}^!G)\]
\end{defi}

\begin{remark}
  In general, computing morphisms between the functors $f_!$ and $g_*$ on $Y$ is hard, and our convolution isomorphism transforms this into a problem with easier functors $\tilde{g}^*$ and $\tilde{f}^!$ on a more complicated space $X\times_Y X'$.
\end{remark}

In what follows, it will be important to be able to study the convolution isomorphism locally. Consider the following diagram, where $j:U\rightarrow Y$ is étale, and all squares are pullbacks:

\[
    \begin{tikzcd}
    X_U\times_U X'_U\arrow[dd,"\tilde{g}_U"]\arrow[rr,"\tilde{f}_U"]\arrow[dr,"\hat{j}"]&&X'_U\arrow[dd,"g_U" near start]\arrow[dr,"\underline{j}'"]&\\
    &X\times_Y X'\arrow[dd,"\tilde{g}" near start]\arrow[rr,"\tilde{f}" near start]&&X'\arrow[dd,"g"]\\
    X_U\arrow[rr,"f_U" near start]\arrow[dr,"\underline{j}"']&&U\arrow[dr,"j"]&\\
    &X\arrow[rr,"f"]&&Y
  \end{tikzcd}\]

The main observation of this section is that the convolution isomorphism is étale local:

\begin{prop}\label{convisolocally}
The following diagram commutes, where the horizontal maps are our convolution isomorphisms, and the vertical maps are restriction followed by base change:
\[\begin{tikzcd}
    \Hom(f_!\_,g_*\_)\arrow[r,"\tau"]\arrow[d]&\Hom(\tilde{g}^*\_,\tilde{f}^!\_)\arrow[d]\\
    \Hom(j^*f_!\_,j^*g_*\_)\arrow[d,"\simeq"]&\Hom(\hat{j}^*\tilde{g}^*\_,\hat{j}^*\tilde{f}^!\_)\arrow[d,"\simeq"]\\
    \Hom({f_U}_!\underline{j}^*\_,{g_U}_*{\underline{j}'}^*\_)\arrow[r,"\tau"]&\Hom({\tilde{g}_U}^*\underline{j}^*\_,{\tilde{f}_U}^!{\underline{j}'}^*\_)
\end{tikzcd}\]
\end{prop}

The verification of this proposition is deferred to Proposition \ref{DIAGRAM} in the Appendix.

\section{Orientations and Duality}
The key aspect of our main theorem (Theorem \ref{thm:main2}) is that for a map $f:X\rightarrow Y$ with smooth source, the dense summand of $f_*\mathbf{1}_X$ on $Y$ is determined by the generic behaviour of the map. To prove this, we will build a comparison morphism for two maps that agree on an open set. The convolution isomorphism of Definition \ref{conviso} turns this problem into giving a suitable morphism $\tilde{g}^*\mathbf{1}_{X}\rightarrow \tilde{f}^!\mathbf{1}_{X'}$. In this section we will describe how to produce such a map via cycle maps and fundamental classes.
\subsection{Topological reminders}
We begin with a leisurely topological reminder of these concepts in the constructible setting. The reader already familiar with this story is invited to skip to \S\ref{Mainarg}, in which we summarise the key idea of the paper.

Let our category $\mathscr{C}$ be that of complex algebraic varieties, and our base change formalism that of Example \ref{runningex}, i.e., $X\mapsto D^b_c(X,k)$. This base change formalism succinctly encodes the topological homology and cohomology of algebraic varieties with coefficients in $k$. We may express the following cohomology groups of $X$ in terms of our functors and the terminal map $t:X\rightarrow \ast$. Letting $\mathbf{1}$ be $k$ on the point $\ast$, we have
\begin{align}\label{Coh derived cat}
H^i(X,k)&\cong \Hom(\mathbf{1},t_*t^*\mathbf{1}[i]),\\
H^i_!(X,k)&\cong \Hom(\mathbf{1},t_!t^*\mathbf{1}[i]),\\
H_i(X,k)&\cong \Hom(\mathbf{1},t_!t^!\mathbf{1}[-i]),\\
H^!_i(X,k)&\cong \Hom(\mathbf{1},t_*t^!\mathbf{1}[-i]).
\end{align}
\begin{remark}
The reader will notice that we are not using the standard notation for
Borel Moore homology (see e.g. \cite{kashiwara2002sheaves}) and compactly supported cohomology. We have opted to use $!$ instead of $BM$ or $c$, as we find this notation more attractive, and it avoids poor notation later
when we discuss more general cohomology theories.
\end{remark}

In this setting, our categories $S_X$ carry some crucial extra structure. They are monoidal and triangulated, with monoidal unit $t^*\mathbf{1}$, and shift functor $[1]$. Throughout, we call $t^*\mathbf{1}$ the $\mathbf{constant}$ $\mathbf{sheaf}$ on $X$, denoted $\mathbf{1}_X$. Similarly, the object $t^!\mathbf{1}$ is the $\mathbf{dualising}$ $\mathbf{sheaf}$ of $X$, denoted $\omega_X$.

\begin{remark}
For a singular space $X$, the dualising sheaf $\omega_X$ is not generally concentrated in a single degree in $D^b_c(X,k)$, so cannot be interpreted as a sheaf in the usual sense.
\end{remark}

\begin{ex}\label{runningex2}
  In our running example of the constructible derived category $X
  \mapsto D^b_c(X,k)$ (Example \ref{runningex})
  $t^*\mathbf{1}$ is the constant sheaf $k_X$. When working in the
  setting of a general base change formalism, we will use $\mathbf{1}_X:= t^*\mathbf{1}$ to denote the unit object, however when dealing with
  sheaves we will often stick to the more standard $k_X$.
\end{ex}

A crucial property of these objects is that for a topological manifold $M$ of dimension $n$, the dualising sheaf $\omega_M$ is locally isomorphic to $\mathbf{1}_M[n]$, a shift of the constant sheaf. To see this, consider the standard triangle associated to the inclusion $j:M\setminus \{x\}\subset M$:
\begin{equation*}
j_!j^!\omega_M\longrightarrow \omega_M\longrightarrow {i_x}_*i_x^*\omega_M\longrightarrow j_!j^!\omega_M[1].
\end{equation*}
In view of \eqref{Coh derived cat}, applying $\Hom(\mathbf{1},t_!\_)$ gives the long exact sequence: \[H_i(M\setminus\{x\})\longrightarrow H_i(M)\longrightarrow \Hom(\mathbf{1},i_x^*\omega_M[-i])\longrightarrow H_{i-1}(M\setminus\{x\})\]
We may therefore identify the $-i^{th}$ cohomology of the stalk of $\omega_M$ at $x$ with the local homology group $H_i(M,M\setminus \{x\})$. Since $M$ is a manifold, it follows by a standard excision argument \cite[\S3.3]{hatcher}  that this sheaf has stalks $k$ concentrated in degree $-n$. By the local homogeneity of manifolds, we see that this sheaf is locally constant, and thus locally isomorphic to $\mathbf{1}_M[n]$.

This also shows that the manifold $M$ is $k$-orientable in the usual
sense (see e.g. \cite{hatcher}, Chapter 3) of having compatible local
generators of these homology groups if and only if we have an
isomorphism $\mathbf{1}_M[n]\cong\omega_M$ in $D^b_c(M,k)$. In view of
the definition of Borel-Moore homology \eqref{Coh derived cat}, this
is equivalent to a class in this group that restricts to a generator
of each local homology group. We call such a class in Borel-Moore
homology a $\mathbf{fundamental}$ $\mathbf{class}$ of $M$. 

There is another, more algebro-topological perspective on orientability. This more general notion of orientability is defined for vector bundles over arbitrary spaces. We say that an $n$-dimensional real vector bundle $V$ over $B$ is orientable with respect to a cohomology theory $E$ if there exists a Thom class $u$ in \[E^n_!(V)\cong E^n(D(V),S(V))\] that restricts to a generator of $E^n(D(V_x),S(V_x))$ for all $x$ in $B$, where $D(V)$ and $S(V)$ denote the associated disk and sphere bundles of $V$.

Thinking about other cohomology theories, there is an analogous base change formalism for algebraic varieties\footnote{There are significant point set topological requirements for the existence of the whole formalism, they are satisfied in our case since our spaces are locally compact and conically stratifiable, and all maps are suitably stratifiable. For a fixed stratification, see Lurie \cite{LurieHA}, and for the functors see Volpe \cite{VolpeM}. We aren't aware of a source for constructibility in this generality, though this should follow same lines as the constructibility proofs in \cite{kashiwara2002sheaves}.} for any cohomology theory represented by an $A_\infty$ ring spectrum $E$ \cite{VolpeM}. This category is the homotopy category of the $\infty$-category of constructible sheaves of $E$-module spectra on $X$, and we will denote\footnote{We have opted for this suggestive notation to encourage the parallel with the constructible derived category.} it by $D^b_c(X,E\perf)$. This category encapsulates the $E$-(co)homology of $X$ in exactly the same manner as $D^b_c(X,k)$ does for $k$-(co)homology. We say a manifold $M$ is $E$-orientable if we may find a fundamental class in \[E_{n}^{!}(M):=\Hom_{D^b_c(M,E\perf)}(\mathbf{1}_M,t^!\mathbf{1}[-n])\] restricting to a generator in all local homology groups $E_n(M,M\setminus \{x\})$. By the previous discussion, we see that $E$-orientability for $M$ is equivalent to the existence of an isomorphism between the $E$ dualising sheaf $\omega_M$ and $\mathbf{1}_M[n]$. An $E$-orientation is then a choice of such an isomorphism $\mathbf{1}_M\rightarrow \omega_M[-n]$. The set of orientations is in general an $\text{Aut}_{E_M}(\mathbf{1}_M)$ torsor, so orientations are not unique.
For a smooth manifold $M$, $E$-orientability in our sense is
equivalent to the Thom class $E$-orientability of the stable normal
bundle of $M$ \cite[Chapter 5, Theorem 2.4]{rudyak}.

\subsection{Orientability and fundamental classes for a general base change formalism.}
With this in mind, let us now work with an arbitrary base change formalism $S$ on the category of algebraic varieties. In order to emphasise the analogy with sheaves we will refer to objects of $S_X$ as sheaves. In addition, we assume the following conditions.

\begin{enumerate}
\item\label{C1} Each $S_X$ is triangulated with shift $[1]$.
\item\label{C2} Over a point, $S_\ast$ has a distinguished object $\mathbf{1}$.
\item\label{C3} For any irreducible $X$ with terminal map $t:X\rightarrow \ast$, the object $t^*\mathbf{1}$ is indecomposable.
\item Topologically proper (resp. \'etale) maps are proper (resp. \'etale) for $S$ in the sense of Definition \ref{basechangeformalism}.
\end{enumerate}
As before, we define the $\mathbf{constant}$ $\mathbf{sheaf}$ and $\mathbf{dualising}$ $\mathbf{sheaf}$:
\begin{align*}\mathbf{1}_X&:=t^*\mathbf{1}\\
\omega_X&:=t^!\mathbf{1}
\end{align*}

These will be the most important objects in what follows.

\begin{defi}\label{Orientable}
  An irreducible variety $X$ of dimension $d$ is $S$-smooth if there exists an isomorphism in $S_X$: 
  \[\mathbf{1}_X\rightarrow \omega_X[-2d].\]
\end{defi}

\begin{remark}
When $S_X=D^b_c(X,\mathbb{Q})$, then $S$-smoothness is the same as rational smoothness. More generally, when $S_X=D^b_c(X,k)$, $S$-smoothness of a variety is the same thing as $k$-smoothness (see e.g. \cite[\S 1]{JWKumar} and \cite[\S 8.1]{FWpsmooth}).
\end{remark}

\begin{remark}
    For a general multiplicative cohomology theory $E$, $X$ is
    $E$-smooth if and only if it is $E$-orientable. By our earlier
    discussion (see \cite[Chapter 5, Theorem
    2.4]{rudyak}), this is
    equivalent to the existence of Thom class in the $E$ cohomology of
    the Thom spectrum of the stable normal bundle of $X$. The Thom
    spectrum perspective helps make the problem of deciding
    $E$-smoothness more concrete and amenable to computation.
\end{remark}

\begin{defi}
  A base change formalism $S$ is \textbf{smoothly orientable} if all smooth irreducible varieties are $S$-smooth.
\end{defi}
\begin{ex}
  The constructible derived category (i.e. $X\mapsto D^b_c(X,k)$) is smoothly orientable. To see this, note that smooth varieties are topological manifolds of twice their algebraic dimension, so it suffices to check that they are $k$-orientable in the usual sense. This can be seen by noting that a smooth manifold is $\mathbb{Z}$-orientable if and only if some transition cocycle of its tangent bundle can be taken to have positive determinant. Since the tangent bundle of $X$ admits an almost complex structure, we can take a presenting cocycle where locally, these transition functions sit inside $GL_n(\mathbb{C})\subset GL_{2n}(\mathbb{R})$. Since $GL_n(\mathbb{C})$ is connected, all of these real matrices have positive determinant, giving the desired orientability.
\end{ex}
The following examples show that deciding smooth orientability can be subtle and geometrically meaningful.

\begin{ex}\label{3.9}
  Let $KU$ denote the spectrum representing the cohomology theory of complex K-theory. Then the base change formalism $X\mapsto D^b_c(X,KU\perf)$ is smoothly
  orientable. To see this, note that any real vector bundle admitting
  a complex structure is $KU$-orientable, via the explicit
  construction of a Thom class in \cite[III, \S~11]{AB-Clifford}.  The
  stable normal bundle of a complex manifold then admits a complex
  structure, so we see that the stable normal bundle of $X$ is $KU$
  orientable, so $X$ is $KU$-smooth.
\end{ex}

\begin{ex}
Let $\mathbb{S}$ denote the sphere spectrum. Then a manifold is
$\mathbb{S}$-orientable if and only if its stable normal bundle admits
a framing, that is, is trivialisable. In particular, most complex
algebraic varieties are not $\mathbb{S}$-smooth. For instance, if $X$
is a smooth surface, the first Pontryagin class of its tangent bundle
is equal to three times its signature, by Hirzebruch's Signature
theorem \cite{Hirzebruch}. So if $X$ has nonzero signature of its
intersection form, e.g., $\mathbb{CP}^2$, then its tangent bundle is
not stably trivial, so neither is its stable normal bundle. In
particular, the base change formalism of sheaves of $\mathbb{S}$-modules on all varieties is not smoothly orientable.\end{ex}

We are interested in invariants of singular varieties, so we want fundamental classes/orientations for singular varieties also. In usual sheaf theoretic fashion, for a Zariski open $j:U\rightarrow X$ in a space $X$, we refer to the functor $j^*\cong j^!$ as restriction to $U$. 

\begin{defi}\label{fundclass}
  An $S$-orientation of an irreducible variety $X$ is a morphism \[\gamma:\mathbf{1}\rightarrow \omega_X[-2d_X]\] which is an isomorphism over the smooth locus of $X$. We say $X$ is orientable with respect to $S$ if an $S$-orientation of $X$ exists.
\end{defi}

The following proposition shows that we can use resolution of singularities to orient all irreducible varieties.

\begin{prop}
If $S$ is smoothly orientable, and $X$ admits a resolution of singularities $f:\tilde{X}\rightarrow X$, then $X$ is orientable.
\end{prop}

\begin{proof}
  Let $f:\tilde{X}\rightarrow X$ be a resolution of singularities. Then $f$ is proper, and $\tilde{X}$ is nonsingular, with $f$ an isomorphism over the smooth locus $U$ of $X$. Since $\tilde{X}$ is smooth, we have an orientation: \[\gamma:\mathbf{1}_{\tilde{X}}\rightarrow \omega_{\tilde{X}}[-2d_X].\] Pushing this forward gives:\[f_*\gamma:f_*\mathbf{1}_{\tilde{X}}\rightarrow f_*\omega_{\tilde{X}}[-2d_X].\]Composing with the unit and counits of our adjunctions gives \[\mathbf{1}_{X}\rightarrow f_*\mathbf{1}_{\tilde{X}} \rightarrow f_*\omega_{\tilde{X}}[-2d_X]\xrightarrow{\text{~}} f_!\omega_{\tilde{X}}[-2d_X]\rightarrow \omega_X[-2d_X].\]
  By base change (BC3), the composite $\gamma_X:\mathbf{1}_X\rightarrow \omega_X[-2d_X]$ restrict to isomorphisms over $U$, giving the desired orientation of $X$.
\end{proof}
\begin{remark}
The definitions of this section are based on purely topological realisation of algebraic varieties via their $\mathbb{C}$-points, but there are natural extensions of these definitions to other settings. For instance, one could consider real pseudomanifolds or algebraic varieties over fields more general than $\mathbb{C}$. In these settings, one would need to modify Definition \ref{Orientable} to reflect the structure at hand. For example, incorporating weights, an orientation is a morphism $\mathbf{1}_X\rightarrow \omega_X[-2d_X](d_X)$, where $(n)$ denotes the Tate twist.
\end{remark}
\begin{remark}
  For the reader who doesn't want to assume resolution of singularities, one can adapt the previous proof to show that if $X$ admits a degree $n$ alteration in the sense of de Jong \cite{deJong}, and $n$ is invertible in the ring $\Hom_{S_\ast}(\mathbf{1},\mathbf{1})$, then $X$ is orientable. 
\end{remark}
The importance of orientations cannot be overstated in our context, since they allow us to produce morphisms between nontrivial objects in our categories $S_X$, via functoriality and the convolution isomorphism.

\begin{defi}\label{BorelMoore}
  For $X$ an algebraic variety, we define the $n^{th}$ compactly supported $S$-homology of $X$ to be \[S^!_{n}(X):=\Hom_{S_X}(1_X,\omega_X[-n]).\]
\end{defi}
This functions similarly to Borel-Moore homology in the constructible setting, as the codomain of a cycle class morphism. In particular, any orientation $\gamma$ of an irreducible variety $X$ is naturally an element of $S^!_{2d_X}(X)$.
Like Borel-Moore homology, these groups are covariantly functorial under proper maps $f:X\rightarrow Y$. This is given by the composition
\[\Hom^*(\mathbf{1}_X,\omega_X)\rightarrow \Hom^*(f_*\mathbf{1}_X,f_*\omega_X)\cong \Hom^*(f_*\mathbf{1}_X,f_!\omega_X)\rightarrow \Hom^*(\mathbf{1}_Y,\omega_Y).\]

\subsection{Why do geometric extensions exist?}\label{Mainarg}
Our goal is to construct a canonical extension of the constant sheaf on a potentially singular variety $Y$. We construct this by first pushing forward the constant sheaf from a resolution of singularities $X\rightarrow Y$. We then need a method for comparing these sheaves for different choices of resolution. We will construct a comparison morphism between these pushforwards using the existence of fundamental classes in our base change formalism.

We may summarise the machinery we have so far for a smoothly orientable base change formalism $S$ as follows.

\begin{itemize}
    \item An $S$ internal notion of smoothness (Definition \ref{Orientable}).
    \item An $S$ orientation/fundamental class for any variety (not necessarily smooth) (Definition \ref{fundclass}).
    \item A compactly supported $S$-homology group to interpret fundamental classes in (Definition \ref{BorelMoore}). \
\end{itemize}

We may now  interpret our convolution isomorphism in this context. Let $X$ and $X'$ be smooth (proper) resolutions of $Y$, with a chosen orientation of $X'$, such that we have a pullback square: 
\begin{equation} 
\begin{tikzcd}
X\times_Y X'\arrow[d,"\tilde{g}"]\arrow[r,"\tilde{f}"]&X'\arrow[d,"g"]\\
  X\arrow[r,"f"]& Y
\end{tikzcd}
\end{equation}
We will use the functoriality of our setup to construct morphisms between $f_!\mathbf{1}_X$ and $g_*\mathbf{1}_{X'}$, from the geometry of fundamental classes on the fibre product. Specifically, the convolution isomorphism and our choice of orientation of $X$ yields the following isomorphism: \begin{align*}
\Hom(f_!\mathbf{1}_X,g_*\mathbf{1}_{X'})&\cong \Hom(\tilde{g}^*\mathbf{1}_X,\tilde{f}^!\mathbf{1}_{X'})\cong\\
&\cong \Hom(\mathbf{1}_{X\times_Y X'},\tilde{f}^!\omega_{X'}[-2d_{X}])=S^!_{2d_{X}}(X\times_Y X').
\end{align*}
Via this isomorphism, we may translate compactly supported $S$-homology classes of $X\times_Y X'$ into maps from $f_!\mathbf{1}_X$ to $g_*\mathbf{1}_{X'}$.

Since $f,g$ are resolutions of $Y$, if $U$ is smooth locus of $Y$, we have a canonical diagonal $\Delta(U)$ inside $X\times_Y X'$, with closure $Z:=\overline{\Delta(U)}$. Choosing an orientation of $Z$, we may push forward the associated fundamental class (\ref{BorelMoore}) to get a class in $S_{2d}^!(X\times_Y X')$. This gives the desired comparison morphism between $f_!\mathbf{1}_X$ and $g_*\mathbf{1}_Y$. 

In the next section, we will show that in the presence of finiteness conditions, we may deduce an isomorphism between the ``dense'' summands of these pushforwards, giving our main theorem. 

\begin{remark}
  In the special case of the constant sheaf, there is an alternate
  argument\footnote{We learnt this argument from Roman Bezrukavnikov.}. Take our base change formalism to be $X\mapsto D^b_c(X,E\perf)$, for a suitably finite (see Definition \ref{KS-definition}) smoothly orientable cohomology theory $E$. One may show (see \cite[Chapter 5, Theorem 2.13]{rudyak}) that for a
map $f:X\rightarrow Y$ of $E$-orientable manifolds, the induced map
$E^*(Y)\rightarrow E^*(X)$ is injective, and upgrade this to the fact
that $\mathbf{1}_Y\rightarrow f_*\mathbf{1}_X$ is split injective in
$D^b_c(X,E\perf)$. Now for
singular $Y$, given two such resolutions $X_i$, we may resolve the diagonal component of
their fibre product $X_1\times_Y X_2$. From this splitting of the
constant sheaf for maps of $E$-orientable smooth manifolds, we see
that the dense summand of $f_{i*}\mathbf{1}_{X_i}$ occurs as a summand of
all resolutions. This argument also shows that isomorphism classes of summands of $f_*\mathbf{1}_X$ over all resolutions $f : X
  \to Y$ form a sort of ``lattice'': given any two resolutions $f_i :
  X_i \to Y$ for $i = 1, 2$, there exists a third resolution $g : Z
  \to Y$ such that all summands of $f_{1*}\mathbf{1}_{X_1}$ and
  $f_{2*}\mathbf{1}_{X_2}$ also occur inside $g_*\mathbf{1}_Z$.
\end{remark}

\section{Finiteness and Krull-Schmidt categories}\label{sec-KS}

In the previous section we constructed a comparison morphism between $f_*\mathbf{1}_X$ and $g_*\mathbf{1}_{X'}$ using an orientation of the irreducible component of the diagonal within $X\times_Y X'$. In any smoothly orientable base change formalism, it follows formally that this comparison morphism is an isomorphism over $U$. In this section, we will introduce the finiteness conditions needed to show that this isomorphism over $U$ lifts to an isomorphism on ``dense summands'' of $f_*\mathbf{1}_X$ and $g_*\mathbf{1}_{X'}$. The finiteness constraint we need is that the categories of the base change formalism are Krull-Schmidt, which allows the use of the crucial Lemma \ref{isolift}.

For completeness, we recall the definition of a Krull-Schmidt category:

\begin{defi}\label{KS-definition}
A category $\mathscr{C}$ is \textbf{Krull-Schmidt} if it is additive with finite sums, and each object is isomorphic to a finite direct sum of indecomposable objects, each with local endomorphism rings. A functor $F$ between Krull-Schmidt categories is Krull-Schmidt if for each indecomposable object $A$, $F$ maps the Jacobson radical of $\End(A)$ into the Jacobson radical of $\End(F(A))$.
\end{defi}

\begin{remark}
There does not appear to be an agreed upon notion of a Krull-Schmidt functor in the literature. The above definition seems to encapsulate the necessary behaviour of our situation.\end{remark}

This condition is easily checked in some sheaf theoretic contexts since it is implied by the following three conditions:

\begin{itemize}
    \item The ``ring of coefficients" $R:=\End(\mathbf{1}_\ast)$ is a complete local ring.
    \item For any $\mathscr{F},\mathscr{G}$ in $S_X$, the group $\Hom_{S_X}(\mathscr{F},\mathscr{G})$ is a finitely-generated $R$-module.
    \item The category $S_X$ has split idempotents.
\end{itemize}

These conditions imply that the endomorphism ring of any indecomposable object is a local $R$-algebra. It also implies that the $R$-algebra morphisms $\End(X)\rightarrow \End(F(X))$ are all finite, giving the Krull-Schmidt property on the functors of the base change formalism.

In particular these conditions are satisfied for the constructible
base change formalism $X\mapsto D^b_c(X,\Lambda)$, when $\Lambda$  is
a field or complete local ring. We will discuss this case in
  more detail in the appendix.

The primary result about Krull Schmidt categories we will use is the following automorphism lifting property.

\begin{lem}\label{isolift}
    Let $F:\mathscr{C}\rightarrow \mathscr{D}$ be a Krull-Schmidt functor between Krull-Schmidt categories, let $A$ be an object with $F(A_i)\neq 0$ for all nonzero summands $A_i$ of $A$, and let $\mu:A\rightarrow A$ an endomorphism of $A$. If $F(\mu)=\text{Id}_{F(A)}$, then $\mu$ is an isomorphism.
\end{lem}

\begin{proof}
    We induct on the number of indecomposable summands of $A$, which is finite by our Krull-Schmidt hypothesis. If $A$ is itself indecomposable, then since $F$ is a Krull-Schmidt functor, $\mu$ doesn't lie in the Jacobson radical of $\End(A)$. Since $\End(A)$ is a local ring, $\mu$ is invertible.

    Next, let us consider the case that $A$ is a direct sum of $n$ copies of a single indecomposable $A_0$. In this case, we see that $\mu-Id_A$ is in the kernel of the algebra morphism $\End(A)\rightarrow \End(F(A))$. This kernel is then contained in the unique maximal two sided ideal $M_{n\times n}(J(\End(A_0))$ of the matrix ring $M_{n\times n}(\End(A_i))\cong \End(A)$. So $\mu$ is in $\Id+J(\End(A))$, and is therefore an isomorphism.
    
    Finally, we may assume that $A$ is not indecomposable, and admits a nontrivial decomposition $A\cong B\oplus C$ where $B$ and $C$ share no isomorphic summands. Then our morphism $\mu$ decomposes as:
\begin{align*}
\mu=\begin{bmatrix}
\mu_{BB} & \mu_{CB}\\
\mu_{BC} & \mu_{CC}
\end{bmatrix}= &\begin{bmatrix}
\mu_{BB} & 0\\
0 & \mu_{CC}
\end{bmatrix}+\begin{bmatrix}
0 & \mu_{CB}\\
\mu_{BC} & 0
\end{bmatrix} \\ & \text{ \quad with $\mu_{XY}\in \Hom(X,Y)$ for $X,Y\in
  \{B,C\}$.}
  \end{align*}
By induction, this diagonal piece is an isomorphism, and since $B$ and
$C$ share no isomorphism classes of summands in common, the second
matrix is in the radical of $\End(A)$ (see the lines following the
proof of \cite[Corollary 4.4]{Krause}), so $\mu$ is an isomorphism.
\end{proof}

\begin{defi}
Let $F:\mathscr{C}\rightarrow \mathscr{D}$ be a Krull-Schmidt functor between Krull-Schmidt categories. Then an object $A$ of $\mathscr{C}$ is $F$-dense if for all summands $A_i$ of $A$ we have $F(A_i)\neq 0$.
\end{defi}
We will only use this notion with respect to $j^*$ for $j:U\rightarrow X$ a Zariski open morphism to irreducible $X$. In sheaf theoretic contexts, this agrees with the usual notion of having all indecomposable summands of dense support, and we will write this as $U$-dense. We say an object $\mathscr{E}$ of $S_X$ is $\mathbf{dense}$ in $S_X$ if it is $U$-dense for any dense Zariski open $U$ of $X$.
\begin{lem}
    Let $F:\mathscr{C}\rightarrow \mathscr{D}$ be a Krull-Schmidt functor between Krull-Schmidt categories. Then any object $A$ of $\mathscr{C}$ has a decomposition $A\cong A_F\oplus A_0$, such that $A_F$ is a maximal $F$-dense summand of $A$. This decomposition is unique up to non-unique isomorphism.
\end{lem}

\begin{proof}
    Choose any decomposition of $A$ into indecomposable objects $A_i$, and let $A_F$ be the summand of those isomorphism types $A_i$ with $F(A_i)\neq 0$. The isomorphism class of $A_F$ is then unique by the Krull-Schmidt property.
\end{proof}

\begin{prop}\label{mutinv}
    Let $A$ and $B$ be objects in a Krull-Schmidt category $\mathscr{C}$ and $F:\mathscr{C}\rightarrow \mathscr{D}$ a Krull-Schmidt functor. If for two maps $f:A\rightarrow B$, $g:B\rightarrow A$, we have $F(f)$ and $F(g)$ are mutually inverse isomorphisms, then $f$, $g$ induce isomorphisms $f'$, $g'$ of $F$-dense summands:
\[\begin{tikzcd}
A_F\arrow[rrr, bend left, "f'"]\arrow[r,"i_A"]&
A\arrow[r,"f"]& B\arrow[r,"\pi_B"] &B_F\end{tikzcd}
\]\[\begin{tikzcd}
B_F\arrow[rrr,bend right,"g'"]\arrow[r,"i_B"]& B\arrow[r,"g"] &A\arrow[r,"\pi_A"]& A_F
\end{tikzcd}\]
     
\end{prop}

\begin{proof}
    First, note that $F(\pi_A)$ and $F(i_A)$ are both isomorphisms, since $A_0$ is sent to zero under $F$ by maximality of $A_F$ and similarly for $B$. Thus, our maps $f':=\pi_B\circ f\circ i_A$ and $g':=\pi_A\circ g\circ i_B$, induce mutually inverse isomorphisms  $F(f')$ and $F(g')$ under $F$.
     So Lemma \ref{isolift} applied to the compositions of these gives that $f'\circ g'$ and $g'\circ f'$ are both isomorphisms. By elementary category theory, this then yields that $f'$ and $g'$ are both isomorphisms, as was to be shown.
\end{proof}

We can now state our main theorem.

\begin{thm}\label{mainthm}
Let $X,X'$ be smooth, irreducible varieties with proper, surjective maps $f:X \rightarrow Y$, $g:X'\rightarrow Y$, and $j:U\rightarrow Y$ a Zariski open in $Y$. Assume that the pullbacks $f_U,g_U:X_U,X'_U\rightarrow U$ are isomorphic over $U$:
\[\begin{tikzcd}[column sep=small]
    X_U\arrow[ddr,"f|_U"']\arrow[drrr]\arrow[rr,dashed,"\simeq"]&&X'_U\arrow[drrr]\arrow[ddl,"g|_U"]\\
    &&&X\arrow[ddr,"f"']&&X'\arrow[ddl,"g"]\\
    &U\arrow[drrr,"j"',hook]&\\
    &&&&Y
\end{tikzcd}\]
Then for any smoothly orientable, Krull-Schmidt base change formalism $S$ the $U$-dense summands of $f_*\mathbf{1}_X$ and $g_*\mathbf{1}_{X'}$ in $S_Y$ are isomorphic.
\end{thm}
\begin{proof}
Let the isomorphism over $U$ be $\alpha:X_U\rightarrow X'_U$. Then we choose orientations of the spaces involved such that we have the following commutative diagram:
\[
\begin{tikzpicture}[xscale=2,yscale=1]
        \node (H0) at (0,0) {$\Hom({f_U}_*\mathbf{1}_{X_U},{g_U}_*\mathbf{1}_{X'_U})$};
        \node (H1) at (0,2) {$\Hom(f_*\mathbf{1}_X,g_*\mathbf{1}_{X'})$};
        \node (S0) at (2,0) {$S_{*,2d}(X_U\times_U X'_U)$};
        \node (S1) at (2,2) {$S_{*,2d}(X\times_Y X')$};
        \draw[->] (H0) to node[above] {$\sim$} (S0);
        \draw[->] (H1) to node[above] {$\sim$} (S1);
        \draw[->] (H1) to (H0);
        \draw[->] (S1) to (S0);
        \node (ur) at (3,2) {$[\overline{\Delta_\alpha}]$};
        \node (e1) at (2.72,2) {$\ni$};
        \node (lr) at (3,-1) {$[\Delta_\alpha]$};
        \node (e1) at (2.72,-.5) {\rotatebox{135}{$\in$}};
        \node (l) at (0,-1) {$\alpha_*\mathbf{1}$};
        \node (e1) at (0,-.6) {\rotatebox{90}{$\in$}};
        \draw[|->] (ur) to (lr);
        \draw[|->] (l) to (lr);
      \end{tikzpicture}
    \]
That we can do this is Proposition \ref{convisocomp} in the appendix. 
Transporting the fundamental class of $\overline{\Delta_\alpha}$ across the convolution isomorphism then yields a map \[f_*\mathbf{1}_X\rightarrow g_*\mathbf{1}_{X'}.\] By commutativity, this restricts to $\alpha_*\mathbf{1}$ over $U$. By symmetry there also exists a morphism back, which gives the two morphisms restricting to mutually inverse isomorphisms over $U$. We may then use Lemma \ref{mutinv} to conclude that these induce isomorphisms on the $U$-dense summands.  \end{proof}

\begin{cor}
    In the setting of Theorem \ref{mainthm}, the dense summands of $f_*\mathbf{1}_X$ and $g_*\mathbf{1}_{X'}$ are isomorphic.
\end{cor}

\begin{proof}
    The dense summands of these are the dense summands of the $U$-dense summands of $f_*\mathbf{1}_X$ and $g_*\mathbf{1}_{X'}$, hence are isomorphic.
\end{proof}
\begin{remark}
    One may note that the use of a \emph{Zariski} neighbourhood was essential in this proof, to be able to take a closure of the graph of the isomorphism over $U$. If one uses a simple étale neighbourhood instead, one must push forward this cycle, and the induced map is an isomorphism only if the degree of the étale morphism is invertible in the ring of coefficients.
  \end{remark}
\section{Applications}
In this final section we will see some applications of Theorem \ref{mainthm}. This theorem allows one to construct canonical objects in $S_X$ for any smoothly orientable base change formalism that play the role of intersection cohomology sheaves in the $\mathbb{Q}$ constructible setting, and parity sheaves in the $\mathbb{F}_p$ constructible setting.

Before considering the general case, let us consider the smoothly orientable base change formalism of constructible sheaves with coefficients in a field or complete local ring $\Lambda$. As an immediate corollary of Theorem \ref{mainthm}, we obtain the following:

  \begin{thm} \label{thm:main1} Let $Y$ be an
    irreducible variety.
    There exists a complex $\mathscr{E}(Y,\Lambda) \in D^b_c(Y,\Lambda)$ characterised up to isomorphism by the following:
    \begin{enumerate}
    \item $\mathscr{E}(Y,\Lambda)$ is indecomposable and its support
      is dense;
    \item $\mathscr{E}(Y,\Lambda)$ is
    a summand inside $f_* \Lambda_X$, for any resolution $f: X \to Y$.
    \end{enumerate}
  \end{thm}

We call $\mathscr{E}(Y,\Lambda)$ the \textbf{geometric extension} on
  $Y$.

As we explained in \S\ref{sec:motivmodrep}, in the special case of a cellular resolution of singularities, this geometric extension will be a parity sheaf \cite{JMW}. We may think of this object as a ``geometrically motivated" minimal way to extend the constant sheaf on the smooth locus of $Y$. In particular, since this summand occurs for any resolution of singularities, we obtain the following corollary for $\mathbb{F}_p$ coefficients.

\begin{cor}\label{ineq}
    For any resolution of singularities $\pi:X\rightarrow Y$, for all $y\in Y$ with fibre $X_y=\pi^{-1}(y)$, we have the inequality \[\dim H^i(\mathscr{E}_{\mathbb{F}_p}(Y)_y)\leq \dim H^i(X_y,\mathbb{F}_p).\]
\end{cor}

\begin{proof}
By definition, we know that $\mathscr{E}_{\mathbb{F}_p}(Y)_y$ is a summand of $i_{y}^*\pi_*\mathbf{1}_X$. The cohomology of $i_{y}^*\pi_*\mathbf{1}_X$ then computes the cohomology of the fibre by proper base change, giving the result.
\end{proof}

We now consider the case of a general smoothly orientable, Krull-Schmidt base change formalism, and higher dimensional local systems. First, we need the definition of a higher dimensional local system in this context.
\begin{defi}
A \textbf{geometric local system} $\mathscr{L}$ on a smooth irreducible variety $U$ is a smooth, proper, surjective map $V\xrightarrow{\mathscr{L}}U$. The restriction of $\mathscr{L}$ to an open $U'\rightarrow U$ is the base change of this morphism.
\end{defi}

The following proposition lets us interpret compactification of morphisms as a method to ``extend'' geometric local systems.

\begin{prop}\label{compactification}
For any geometric local system $V\xrightarrow{\mathscr{L}} U$ over $U$ a (smooth Zariski) open in $Y$, there exists a proper morphism $X\xrightarrow{\tilde{\mathscr{L}}} Y$ from smooth $X$ such that we have a pullback square:
\[\begin{tikzcd}
    V\arrow[r]\arrow[d,"\mathscr{L}"]& X\arrow[d,"\overline{\mathscr{L}}"]\\
    U\arrow[r,"j"]&Y
\end{tikzcd}\]
\end{prop}

\begin{proof}
    The proof may be summarised in the following diagram:\[\begin{tikzcd}
    &X\arrow[d]\arrow[dd,bend left,"\overline{\mathscr{L}}"]\\ V\arrow[d,"\mathscr{L}"]\arrow[ur,dashed]\arrow[r]&\tilde{Y}\arrow[d,"g"']\\
    U\arrow[r,"j"]&Y
    \end{tikzcd}\]    
    First, we compactify the composition $j\circ \mathscr{L}$ as
    $V\rightarrow \tilde{Y}\xrightarrow{g} Y$, where $g$ is
    proper. Choosing a resolution of singularities $X\rightarrow
    \tilde{Y}$, as $V$ is smooth, the map $V\rightarrow Y$ factors through $X$. Then composing with $g$ gives the desired map
    $X\xrightarrow{\tilde{\mathscr{L}}} Y$. \end{proof}

From here we will let $S$ denote a Krull-Schmidt, smoothly orientable base change formalism, satisfying the following condition for geometric local systems:

\begin{itemize}
\item[(D)] If $\mathscr{L}:V\to U$ is a geometric local system, then all summands of $\mathscr{L}_*\mathbf{1}$ have dense support.
\end{itemize}
\begin{remark}
This condition holds in all examples we have discussed so far, and we don't know of any situation where it fails to hold. In sheaf theoretic or algebro-topological situations, this follows from homotopy invariance.
\end{remark}
  
With these preliminaries, we have the following general version of Theorem \ref{thm:main1}.

\begin{thm} 
 \label{thm:main2} Let $Y$ be an irreducible variety, $S$ a smoothly orientable, Krull-Schmidt base change formalism satisfying condition (D). For any dense $U \subset Y$ and geometric local system $V
    \xrightarrow{\mathscr{L}}U$ there is a unique object
    $\mathscr{E}(Y,\mathscr{L}) \in S_Y$  satisfying:
    \begin{enumerate}
    \item $j^*\mathscr{E}_S(Y, \mathscr{L}) \cong
      \mathscr{L}_*\mathbf{1}_V$ where $j : U \hookrightarrow Y$ denotes
      the inclusion;
      \item      $\mathscr{E}_S(Y, \mathscr{L})$ is dense, with no summands supported on a proper closed subset of $Y$;
    \item for any proper map with smooth source $f : X \to Y$ which restricts to
      $\mathscr{L}$ over $U$, $\mathscr{E}_S(Y,\mathscr{L})$ occurs as a
      summand of $f_* \mathbf{1}_X$. 
    \end{enumerate}
\end{thm}

\begin{defi}
We define the \textbf{geometric extension} of $\mathscr{L}$ on $Y$ to be the object $\mathscr{E}_S(Y,\mathscr{L})$. When the local system is the identity, we call this the geometric extension on $Y$. We call the groups
\[\mathscr{E}_S^i(Y):=\Hom_{S_Y}(\mathbf{1}_Y,\mathscr{E}_{S}(Y)[i])\]
the geometric $S$-cohomology groups of $Y$.
\end{defi}

\begin{remark}
For a fixed $S$, one may replace smoothness with $S$-smoothness (see Definition \ref{Orientable}) in the preceeding definitions, with slightly more general results.
\end{remark}

One can think of these groups $\mathscr{E}^i_S(Y)$ concretely as unavoidable summands of the $S$-cohomology of any resolution of singularities of $X$.

\begin{warning}
In general, the object $\mathscr{E}_S(Y,\mathscr{L})$ depends on the geometry of the map $V\xrightarrow{\mathscr{L}} U$, not just on the object $\mathscr{L}_*\mathbf{1}_V$ in $S_U$. An example with further discussion is given in Example \ref{ex:legendrefamily}.
\end{warning}
We will now give some properties of these geometric extensions.

\begin{prop}
    If $\mathscr{L}$ and $\mathscr{L}'$ on $U$ and $V$ agree on $U\cap
    V$, then: \[\mathscr{E}_S(Y,\mathscr{L})\cong
      \mathscr{E}_S(Y, \mathscr{L'}).\]
    In particular, for $f:X\to Y$ a proper map with smooth source, the
    dense summand of $f_*\mathbf{1}_X$ depends only on the generic
    behaviour of the map $f$.
\end{prop}
\begin{proof}
    The geometric extensions arise as summands of $f_*\mathbf{1}_X$, $g_*\mathbf{1}_{X'}$ for compactifications $f,g$ of these geometric local systems. Since these two maps agree on a dense open $U\cap V$, their dense summands are isomorphic by Theorem \ref{mainthm}.
\end{proof}
Consider a complex $\mathscr{L}$ isomorphic to $\bigoplus_i \mathscr{L}^i[-i]$ for local systems $\mathscr{L}^i$ on a dense subvariety of the smooth locus of $Y$. Define $IC(Y,\mathscr{L})$ to be the sum $\bigoplus_i IC(Y,\mathscr{L}^i)[-i]$ (see Remark \ref{ICdef}).
\begin{prop}\label{ICprop}
    If $S$ is the constructible derived category of sheaves over $\mathbb{Q}$, then the geometric extension is the intersection cohomology complex of sheaves:
    \[\mathscr{E}_S(Y,\mathscr{L})\cong IC(Y,\mathscr{L}_*\mathbb{Q}).\]
\end{prop}

\begin{proof}
    By the Decomposition Theorem \cite{BBD}, the pushforward
    $f_*\mathbf{1}_X$ is a direct sum of semisimple perverse sheaves. On
    the smooth locus of $Y$, this sheaf is the local system
    $f_*\mathbf{1}_V$. By the classification of simple perverse sheaves
    (\cite{BBD} or \cite[Theorem 3.4.5]{Achar}), we see that the dense summand of $f_*\mathbf{1}_V$ is $IC(Y,\mathscr{L}_*\mathbf{1}_V)$.
\end{proof}

In general, like intersection cohomology sheaves, the geometric extension gives a way to interpolate between $S$-cohomology and noncompact $S$-homology.

\begin{prop}\label{interpolate}
    For any resolution of singularities $f:\tilde{Y}\rightarrow Y$, chosen orientation $\gamma$ of $\tilde{Y}$, and choice of split inclusion $\mathscr{E}_S(Y)\rightarrow f_*\mathbf{1}_{\tilde{Y}}$, we obtain a sequence: \[\mathbf{1}_Y\rightarrow \mathscr{E}_S(Y)\rightarrow \omega_Y[-2d_Y]\] such that the composite is an orientation of $Y$.
\end{prop}

\begin{proof}
  We have the following maps of sheaves on $Y$:
  \[\mathbf{1}_Y\rightarrow f_*\mathbf{1}_{\tilde{Y}}\rightarrow \mathscr{E}_S(Y)\rightarrow f_*\mathbf{1}_{\tilde{Y}}\cong f_!\mathbf{1}_{\tilde{Y}}\cong f_!\omega_{\tilde{Y}}[-2d]\rightarrow \omega_Y[-2d]\]
  These maps are all isomorphisms over $U$, so the composite is an orientation of $Y$.
\end{proof}
The following shows that for $S$-smooth varieties, the geometric extension is just the constant sheaf.
\begin{prop}
    If $S$ is a smoothly orientable base change formalism, and $Y$ is $S$-orientable, then the geometric extension is the constant sheaf on $Y$:\[\mathscr{E}_S(Y)\cong \mathbf{1}_Y\]
\end{prop}

\begin{proof}
For a resolution $\pi:\tilde{Y}\rightarrow Y$ of $Y$, and chosen
isomorphism
$\gamma:\mathbf{1}_{\tilde{Y}}\rightarrow\omega_{\tilde{Y}}[-2d]$ of
$\tilde{Y}$, then we claim that the pushforward orientation of $Y$
is an isomorphism $\mathbf{1}_Y\rightarrow \omega_Y[-2d]$. This
morphism is between indecomposable, isomorphic objects, and is thus
an isomorphism over $U$, so is an isomorphism by the Krull-Schmidt
property. Composing with the inverse isomorphism
$\omega_Y[-2d]\rightarrow \mathbf{1}_Y$ then gives the following
commutative diagram: 
\[\begin{tikzcd}
    \mathbf{1}_Y\arrow[r]\arrow[d,equal]&\mathscr{E}_S(Y)\arrow[d]\\
    \mathbf{1}_Y&\omega_Y[-2d]\arrow[l].
\end{tikzcd}\]
The result then follows from Proposition \ref{mutinv}.
\end{proof}

The maps of Proposition \ref{interpolate} induce the following interpolation morphisms
\[S^*(Y)\rightarrow \mathscr{E}_S^*(Y)\rightarrow S^!_{2d_Y-\ast}(Y).\]

This interpolation perspective also lets us extract a canonical invariant of our singular space, the (co)kernel of the induced map $S^*(Y)\to \mathscr{E}^*_S(Y)$. 
\begin{defi}
    Let $Y$ be irreducible and projective. The \textbf{geometrically pure} $S$-cohomology of $Y$ is the quotient \[S^*_{gp}(Y):=\frac{S^*(Y)}{\text{ker}(S^*(Y)\rightarrow \mathscr{E}^*_S(Y))}\]
    Similarly, the \textbf{geometrically non-pure} $S$-cohomology of $Y$ is this kernel \[S^*_{gnp}(Y):=\text{ker}(S^*(Y)\rightarrow \mathscr{E}^*_S(Y)).\]
\end{defi}
That these objects are independent of the choices involved in their construction is the content of the following Lemma.

\begin{lem}\label{comp}
    Let $Y$ be irreducible, with two resolutions of singularities \[f_i:\tilde{Y}_i\rightarrow Y\ \ \  \text{for}\  i\in \{1,2\}. \]
    Assume for each map we have a chosen a split projection onto the geometric extension of $Y$: \[{f_i}_*\mathbf{1}\xrightarrow{\pi_i}\mathscr{E}_S(Y).\]
    Then there exists an isomorphism $\beta$ of $\mathscr{E}_S(Y)$ such that the following diagram commutes:
    \[\begin{tikzcd}
        \mathbf{1}\arrow[dr]\arrow[r]& {f_1}_*\mathbf{1}\arrow[r,"\pi_1"]&\mathscr{E}_S(Y)\arrow[d,dashed,"\beta"]\\
    &{f_2}_*\mathbf{1}\arrow[r,"\pi_2"]&\mathscr{E}_S(Y)
    \end{tikzcd}\]
\end{lem}

\begin{proof}
   Resolving the diagonal irreducible component of the fibre product $Y_1\times_Y Y_2$, we may find a third resolution of singularities of $Y$, dominating $f_i$:\[\begin{tikzcd}
       \tilde{Y}_3\arrow[d,"f_3"]\arrow[r] &\tilde{Y}_1\times_Y \tilde{Y}_2\arrow[dl] \\Y
   \end{tikzcd}\]
   Then for any choice of split projection ${f_3}_*\mathbf{1}\rightarrow \mathscr{E}_S(Y)$, we have the following commutative diagram for $i\in \{1,2\}$:
   \[\begin{tikzcd}
       \mathbf{1}\arrow[r]&{f_i}_*{\mathbf{1}}\arrow[r]&{f_3}_*\mathbf{1}\arrow[d]\\
       &\mathscr{E}_S(Y)\arrow[u]\arrow[r,dashed,"\beta_i"]&\mathscr{E}_S(Y)
   \end{tikzcd}\]
   These maps $\beta_i$ defined as the composition are isomorphisms by
  Theorem \ref{thm:main1}, so their composite $\beta_2^{-1}\circ \beta_1$ gives the desired isomorphism.
\end{proof}

\begin{remark}
In the case of $\mathbb{Q}$-constructible coefficients, the geoemtrically pure (resp geometrically non-pure) is
precisely the pure (resp non-pure) cohomology in the mixed Hodge structure on
$H^*(Y,\mathbb{Q})$. To see this, recall that the mixed Hodge
structure on a singular, projective variety is given by resolving $Y$ by
a smooth simplicial hypercover, and the pure component is the first
quotient of the associated spectral sequence \cite{DeligneHodgeIII}.
\end{remark}

\begin{ex} \label{ex:A/pm}
Let us consider the geometric extension on the space $\mathbb{A}^n_{\mathbb{C}}/\pm 1 $ with constructible coefficients over a field $k$ of characteristic two. We will show that the geometric extension over $k$ is exactly $\pi_*\mathbf{1}$ for a resolution $\pi$ that contracts a divisor over $0$. We will thus have nonzero cohomology in degree $2(n-1)$ in the stalk over $0$. This gives the geometric consequence that any resolution of singularities of this space must contract a divisor, by Corollary \ref{ineq}.\footnote{As explained to us by Burt Totaro, this may also be easily seen algebro-geometrically by the fact that our space $\mathbb{A}^n/\pm 1$ is $\mathbb{Q}$-factorial, as follows. Let $X\xrightarrow{\pi} \mathbb{A}^n/\pm 1$ be a resolution of singularities, and $D$ a chosen very ample Weil divisor on $X$. As $\mathbb{A}^n/\pm 1$ is $\mathbb{Q}$ factorial, a positive multiple $n\pi_*(D)$ of the Weil divisor $\pi_*(D)$ is Cartier. Pulling this back gives the Cartier divisor $\pi^*(n\pi_*(D))$ on $X$. If our exceptional fibre has codimension at least $2$, then this divisor on $X$ would be $nD$, but then $D$ cannot be very ample, as its sections don't separate points in the exceptional fibre.}

Consider the following diagram of blowups and quotients:

\[\begin{tikzcd}
\text{Bl}_0(\mathbb{A}_\mathbb{C}^n)\arrow[r,"\simeq"]\arrow[d,"\tilde{\pi}"]\arrow[dr,"\simeq"]&\mathbb{P}_\mathbb{C}^{n-1}\\
\mathbb{A}_\mathbb{C}^{n}\arrow[dr]&\text{Bl}_0(\mathbb{A}_\mathbb{C}^{n})/\pm 1\arrow[u,"\simeq"]\arrow[d,"\pi"]&\mathbb{P}_\mathbb{C}^{n-1}\arrow[l]\arrow[d]\\
&\mathbb{A}_\mathbb{C}^{n}/\pm 1&\{0\}\arrow[l]
\end{tikzcd}\]

The space $\text{Bl}_0(\mathbb{A}_\mathbb{C}^n)$ via its projection map to $\mathbb{P}_\mathbb{C}^{n-1}$ is the total space of the tautological bundle, and this quotient $\text{Bl}_0(\mathbb{A}_\mathbb{C}^{n})/\pm 1$ is obtained by taking the quotient under the inversion map on the (vector space) fibres. So we see the maps with $\simeq$ are homotopy equivalences of topological spaces, and $\text{Bl}_0(\mathbb{A}_\mathbb{C}^n)/\pm 1$ is smooth. So we obtain a resolution of the singular space $\mathbb{A}_\mathbb{C}^n/\pm 1$, with domain homotopic to $\mathbb{P}_\mathbb{C}^{n-1}$. On our base, $0$ is the unique singular point, and the fibre over this singular point in this resolution is $\mathbb{P}_\mathbb{C}^{n-1}$. Its complement is $\mathbb{A}_\mathbb{C}^n-\{0\}/\pm1$, which is naturally homeomorphic to the space $\mathbb{RP}^{2n-1}\times \mathbb{R}$.

We claim that the geometric extension is just the pushforward $\pi_*\mathbf{1}$. To show this, we need to check that this sheaf is indecomposable, which is equivalent to showing that it has no skyscraper summands at the singular point.

To check this, consider the compactly supported cohomology of the open-closed triangle for the inclusion of the singular point:\[j_!j^!\pi_*\mathbf{1}\rightarrow \pi_*\mathbf{1}\rightarrow i_*i^*\pi_*\mathbf{1}\xrightarrow{+1}\]

By base change, the ompactly supported cohomology of $j_!j^!\pi_*\mathbf{1}$ is the compactly supported cohomology of $\mathbb{A}_\mathbb{C}^n-\{0\}/\pm1\simeq \mathbb{RP}^{2n-1}\times \mathbb{R}$. Similarly, by base change, the sheaf $i_*i^*\pi_*\mathbf{1}$ computes the cohomology of the fibre, which is $\mathbb{P}_\mathbb{C}^{n-1}$. The middle term computes the compactly supported cohomology of the total space, which is a complex line bundle over $\mathbb{P}_\mathbb{C}^{n-1}$, and so gives the cohomology of $\mathbb{P}^{n-1}_\mathbb{C}$, shifted by $2$. Applying the compactly supported cohomology functor $\Hom^*_{k}(\mathbf{1},t_!\_)$ yields an exact triangle:
\[\begin{tikzcd}
    H^*_!(\mathbb{RP}^{2n-1}\times \mathbb{R},k)\arrow[r]&H^{*-2}(\mathbb{P}_\mathbb{C}^{n-1},k)\arrow[d]\\
    &H^*(\mathbb{P}_\mathbb{C}^{n-1},k)\arrow[ul,"+1"]
\end{tikzcd}\]

Since the characteristic of $k$ is two, $H^*_!(\mathbb{RP}^{2n-1}\times \mathbb{R})$ is nonzero in all degrees between $1$ and $2n$ inclusive. So this $+1$ degree map must be injective by the parity of the cohomology of $\mathbb{CP}^{n-1}$. Thus, this extension is maximally nonsplit and there can be no skyscraper sheaf summand. So this $\pi_*\mathbf{1}$ is indecomposable in characteristic two, giving the desired nonzero cohomology in the stalk. (This may also be seen using intersection forms (see \cite[\S \S3.2-3.3]{JMW}). The refined intersection form is identically zero modulo 2.)

\end{ex}

The previous example shows that geometric extensions for the constructible base change formalism over fields need not be a perverse, and by similar ideas we obtain the following geometric consequence.

\begin{prop} \label{prop:no-semi-small}
    If for some field $k$, the geometric extension of $Y$ over $k$ is not perverse up to shift in $D^b_c(Y,k)$, then $Y$ does not admit a semismall resolution. 
\end{prop}

\begin{proof}
    For such an $Y$, the geometric extension is a summand of $\pi_*\mathbf{1}_{\tilde{Y}}$ for any resolution $\pi:\tilde{Y}\rightarrow Y$. Since the geometric extension is assumed to not be perverse, no resolution can be semismall.
\end{proof}

The following is an immediate corollary of Theorem \ref{thm:main2}, though we suspect there is a more direct way to see the result.

\begin{prop}[Zariski trivial is cohomologically trivial.]
    Let $f:X\rightarrow Y$ be a smooth proper morphism between smooth varieties. Then if $f$ is Zariski locally trivial, then $f_*\mathbf{1}_X$ is the trivial local system on the fibre for any smoothly orientable base change formalism.
\end{prop}

\begin{proof}
Let $F$ be a fibre of this morphism. Then if $f^{-1}(U)\cong F\times U$, then $f_*\mathbf{1}_X$ is isomorphic to the geometric extension of the constant $F$ local system over $U$. But the trivial family $F\times X\rightarrow X$ also gives the geometric extension, giving the result.
\end{proof}
 \begin{remark}
 By a similar argument, if $f$ is étale locally trivial, then $f_*\mathbf{1}$ is trivial in any base change formalism with coefficients of characteristic zero.
 \end{remark}

By Proposition \ref{ICprop}, over $\mathbb{Q}$, the geometric
extension $\mathscr{E}_{\mathbb{Q}}(Y,\mathscr{L})$ is determined by
the $\mathbb{Q}$ local system $\mathscr{L}_*\mathbf{1}_V$ on $U$ within
$Y$, being isomorphic to $IC(Y,\mathscr{L}_*\mathbf{1}_V)$. The
following example shows that this is exceptional behaviour, and that
geometric extensions in general are not determined by the $S$-local
systems $\mathscr{L}_*\mathbf{1}_V$. They require the map.

\begin{ex}[The Legendre family of elliptic curves]\label{ex:legendrefamily}
    Consider the following projective family $E_t$ of elliptic curves \[\begin{tikzcd}        X\arrow[d,"\pi"]\arrow[r]&\mathbb{A}^1_{\mathbb{C}}\times \mathbb{P}_{\mathbb{C}}^2\arrow[dl,"\text{pr}"]\\
    \mathbb{A}^1_{\mathbb{C}}
    \end{tikzcd}
    \]
    given by: \[Y^2Z=X(X-Z)(X-tZ).\]
    Here $t$ is the coordinate on $\mathbb{A}^1_\mathbb{C}$, and we view this family inside $\mathbb{A}^1_\mathbb{C}\times\mathbb{P}^2_{\mathbb{C}}$. This family is smooth away from $t\in\{0,1\}$. The total space of this family has two isolated singular points, the nodes of the nodal cubics in the fibres over $t=0$ and $t=1$. These singular points have tangent cone isomorphic to the cone on a smooth conic. Blowing up these two singular points resolves the singularities, giving the resolved family \[\tilde{X}\xrightarrow{\tilde{\pi}}\mathbb{A}^1_\mathbb{C}.\]
    For this new map, the fibres over $t\in\{0,1\}$ are each the union
    of two rational curves intersecting in two points transversely. (This is type $I_2$ in Kodaira's classification of elliptic fibres
    \cite{Kodaira}. This is often called the ``double banana'' 
    configuration.) So for the constructible base change formalism with coefficients in $k$, the stalks $i^*_t$ of $\tilde{\pi}_*\mathbf{1}$ are given by:
\[\begin{tabular}{ p{2.5cm}||p{1cm} p{1cm} p{1cm}  }

 $H^*$& 0 &1&2\\
 \hline
 $i_0^*\tilde{\pi}_*\mathbf{1}$& $k$   &$k$&$k^{\oplus 2}$\\
 $i_1^*\tilde{\pi}_*\mathbf{1}$&  $k$  &$k$   &$k^{\oplus 2}$\\
$i_t^*\tilde{\pi}_*\mathbf{1}$ if $t\neq 0,1$&$k$ & $k^{\oplus 2}$&  $k$\\
\end{tabular}\]

The monodromy of this family is nontrivial only in the middle
degree: \[H^1(E_t,k)\cong k^2.\] One may then compute (see
e.g. \cite[Part 1]{CMP}) that the monodromy of a small loop around either singular fibre over $\mathbb{Z}$ is similar to \[\begin{bmatrix}
1 & 2\\
0 & 1
\end{bmatrix}.\]
Thus, we observe that this monodromy action is trivial if the characteristic of $k$ is two, and in this case the associated local system on $\mathbb{A}^1_\mathbb{C}\setminus \{0,1\}$ is trivial. We note however that the geometric extension of this local system is always nontrivial, as a trivial local system $E\times \mathbb{A}^1_\mathbb{C}$ has two copies of $k[-1]$ in its stalk over its singular fibres, rather that the one copy in our family.

This example therefore shows that the geometric extension of a geometric local system cannot be deduced from just the knowledge of $\tilde{\pi}_*\mathbf{1}$ restricted to the open subset $\mathbb{A}^1_\mathbb{C}\setminus \{0,1\}$. This example also lets us observe the failure of the local invariant cycle theorem in characteristic $p$, as the specialisation map \[H^1(E_0)\rightarrow H^1(E_t)^\mu\] is not surjective.
\end{ex}
\begin{remark}
One may construct families of counterexamples as follows. Let  $\pi:X\rightarrow \mathbb{A}^1$ be proper with smooth source, smooth over $\mathbb{A}^1\setminus \{0\}$, such that the $n^{th}$ power base change of $\pi$ has smooth total space $\tilde{X}_n$.
\[\begin{tikzcd}    \tilde{X}_n\arrow[r]\arrow[d,"\pi_n"]&X\arrow[d,"\pi"]\\
    \mathbb{A}^1\arrow[r,"z\mapsto z^n"]&\mathbb{A}^1
\end{tikzcd}\]
Then the associated monodromy representation of $\pi_n$ has monodromy
the $n$th power of the monodromy of $\pi$. This allows one to
trivialise the monodromy if the order is finite, dividing $n$.
\end{remark}
Another feature of the decomposition theorem in the $\mathbb{Q}$ coefficient setting is that the pushforward $f_*\mathbf{1}_X$ is semisimple. This fails for more general coefficients, and can already be seen with smooth, projective maps with mod $p$ coefficients. This result and its proof don't require geometric extensions, but we've decided to include it as we are not aware of any examples of this phenomenon in the literature.

\begin{ex}[A non-semisimple geometric local system] \label{ex:nonss}
Let $S$ be the $\mathbb{F}_2$ constructible formalism, and let
$\pi:E\rightarrow X$ be an algebraic (étale local)
$\mathbb{P}^1_{\mathbb{C}}$ bundle over a smooth space $X$, with
nontrivial Brauer class in $H^3(X,\mathbb{Z})$. For instance, one may
take the tautological bundle over an algebraic approximation of the
classifying space $BPGL_2(\mathbb{C})$ (see e.g. \cite{BL}). Then $\pi_*\mathbf{1}_E$ is an extension of $\mathbf{1}_X$ by $\mathbf{1}_X[2]$, classified by an element of \[\text{Ext}^1(\mathbf{1}_X,\mathbf{1}_X[2])=H^3(X,\mathbb{F}_2).\] This element is the reduction modulo $2$ of the associated Brauer class in $H^3(X,\mathbb{Z})$, so does not vanish in this quotient, and gives the desired indecomposable local system. One may construct counterexamples more generally using the fact that if $f:X\rightarrow Y$ is any map, and the induced map $H^*(Y,\mathbb{F}_p)\rightarrow H^*(X,\mathbb{F}_p)$ is not injective, then $\mathbf{1}\rightarrow f_*\mathbf{1}$ cannot be split injective in $D^b_c(X,\mathbb{F}_p)$.
\end{ex}

Thus far in this section we have been considering applications for the
constructible derived category with field coefficients, but it is
worth emphasising that there are other examples, which have been shown
to be relevant to geometric representation theory.\footnote{The most
  famous examples are Kazhdan and Lusztig's computation of the
  equivariant K-theory of the Steinberg variety \cite{KL} and
  Nakajima's computation of the equivariant K-theory of quiver
  varieties \cite{Nakajima}. Note that both computations can be
  interpreted as the computation of an endomorphism of a direct image
  with K-theory coefficients.}

In particular, there is now an established theory, with a six functor formalism, for modules over any $A_\infty$ ring spectrum. Let's consider the ring spectrum $KU_p$, $p$ completed complex K theory. This is the ring spectrum that represents the cohomology theory $X\mapsto K^*(X)\otimes_\mathbb{Z} \mathbb{Z}_p$ on finite $CW$ complexes $X$, where $K^0(X)$ is the usual Grothendieck group of complex vector bundles on $X$. The formalism of $KU_p$ modules then gives rise to a smoothly orientable, Krull-Schmidt base change formalism for $p$ completed complex K theory $KU_p$, see Appendix \ref{app.KUp}.

By Theorem \ref{thm:main2}, we may then define the geometric extension for $p$ completed K theory.

\begin{defi}\label{defKtheory}
For an irreducible variety $Y$, the $KU_p$ geometric extension is the sheaf of $KU_p$ modules $\mathscr{E}_{KU_p}(Y)$. The geometric $K$-theory groups at $p$ are defined to be the homotopy groups of this indecomposable sheaf of $KU_p$ modules:\[\mathscr{E}^{*}_{KU_p}(Y):=\pi_{-*}(\mathscr{E}_{KU_p}(Y)).\]
\end{defi}

We end with some natural questions regarding these geometric K groups.

\begin{question}
    We have natural maps $K^*(Y)\rightarrow \mathscr{E}^*_{KU_p}(Y)$ for all $p$, and by Lemma \ref{comp}, the kernel of these maps are independent of our choices. We might call elements in this common kernel nonpure classes in (integral) K-theory. Is there a geometric, vector bundle description of these classes? The nontorsion part will be visible as nonpure classes in $\mathbb{Q}$ cohomology, what about the torsion?
\end{question}

One may also use the rationalised K theory spectrum $KU_\mathbb{Q}$ in the preceding definitions. In this case, the groups we obtain are just ordinary intersection cohomology, since in rational cohomology, geometric extensions are just intersection cohomology, and the Chern character gives an isomorphism of $E_\infty$ ring spectra $\text{ch}:KU_\mathbb{Q}\cong H_\mathbb{Q}$.
Though the groups are not new, this alternate description of intersection cohomology via K theory leads to the natural question of whether these geometric extensions can be categorified. Our main result gives, for a fixed base space $Y$, for any resolution $X\rightarrow Y$, an idempotent endomorphism of $K^*(X)\otimes \mathbb{Q}$ which cuts out $\mathscr{E}_{KU_\mathbb{Q}}(Y)$, and this image is independent of the resolution. Now specialise to the case where $X$ and $Y$ admit compatible affine pavings, so $X\times_Y X$ also admits such a paving. This occurs for instance in the theory of Schubert varieties with Bott-Samelson resolutions. In this situation, the rationalised Grothendieck group of coherent sheaves on $X$ is isomorphic to the rationalised topological K group of $X$, via the Chern character to Chow groups. The fundamental classes of subvarieties of $X\times X$ naturally act as endomorphisms of $D^b_{coh}(X)$ as kernels of Fourier-Mukai transforms, and this naturally categorifies the action on $K^*(X)$. This leads to the following (imprecise) question.

\begin{question}
    Let $X$ be an affine paved resolution of $Y$. Does there exist an idempotent endofunctor $E_{X/Y}$ of $D^b_{coh}(X)$ such that the image of $E_{X/Y}$ is an invariant of $Y$, which decategorifies to the idempotent cutting out $\mathscr{E}_{KU_\mathbb{Q}}(Y)$ inside $K^*_\mathbb{Q}(X)$? Furthermore, is this category independent of the resolution $X$?
\end{question}

\section{Appendix}

\subsection{$KU_p$ modules}\label{app.KUp}

Here we will give a short introduction to the smoothly orientable base change formalism of $KU_p$ modules. We will start with ordinary topological K theory. This is a cohomology theory built from the Grothendieck group of complex vector bundles over $X$:
\[X\mapsto K^0(X):=Gr(Vec_{\mathbb{C}}/X).\]
This is the zeroth of a series of functors $K^i$, which form a cohomology theory in the sense of satisfying the Eilenberg-Steenrod axioms (except the dimension axiom).
Bott's periodicity theorem (\cite{bott1959stable}, also see \cite{hatcherVB}) implies that this cohomology theory is represented by a two periodic sequential spectrum $KU$ with component spaces:
\begin{align*}
    KU_{2i}&\cong \mathbb{Z}\times BU\\
    KU_{2i+1}&\cong U.
\end{align*}
Here $U$ is the infinite unitary group, and $BU$ is the union of the infinite complex Grassmannians $BU(n)$. The tensor product on vector bundles gives a homotopy coherent commutative multiplication law on this spectrum, so $KU$ is naturally an $E_\infty$-ring.

The (higher) coherence of this multiplication law allows one to define a well-behaved $\infty$-category of module spectra. This $\infty$-category is stable, so it can be thought of as an enhancement of its triangulated homotopy category.

For any stable $\infty$-category $\mathcal{C}$, we have a notion of sheaves on a space $X$ valued in $\mathcal{C}$. We will not define this precisely, but in rough terms it gives an object for each open set, a morphism for each inclusion of open sets, a homotopy between the compositions for each pair of composable inclusions, and so on, such that an analogue of the sheaf condition holds.

The $\infty$-category of such $\mathcal{C}$-valued sheaves on a space $X$ is stable. If one restricts to suitable, locally compact spaces with well-behaved maps between them, such as algebraic varieties with algebraic maps, then we obtain the whole six functor formalism for $\mathcal{C}$-valued sheaves, see e.g. \cite{VolpeM}. Furthermore, one may restrict to constructible $\mathcal{C}$-valued sheaves. Constructibility is then preserved under these six functors, due to the good topological properties of algebraic maps. We will not need the inner workings of this construction. 
The following example shows why such a formal black box can still be useful.

\begin{ex}
Consider Example \ref{ex:A/pm}, interpreted within the K-theoretic framework. This whole example is formal, until we apply compactly supported cohomology to obtain the triangle:
\[\begin{tikzcd}
    H^*_c(\mathbb{RP}^{2n-1}\times \mathbb{R},k)\arrow[r]&H^{*-2}(\mathbb{P}_\mathbb{C}^{n-1},k)\arrow[d]\\
    &H^*(\mathbb{P}_\mathbb{C}^{n-1},k)\arrow[ul,"+1"]
\end{tikzcd}\]
If we instead used K theory, we would obtain the triangle: 
\[\begin{tikzcd}
    K^*_c(\mathbb{RP}^{2n-1}\times \mathbb{R})\arrow[r]&K^{*-2}(\mathbb{P}_\mathbb{C}^{n-1})\arrow[d]\\
    &K^*(\mathbb{P}_\mathbb{C}^{n-1})\arrow[ul,"+1"]
\end{tikzcd}\]
Then one may show formally that the vertical arrow is multiplication by the Thom class of the $KU$ orientable line bundle $\mathscr{O}(2)$, and that compactly supported K-theory of a compact space times $\mathbb{R}$ is the ordinary K-theory shifted by one. 
Since the Thom class is $1-2[H]$, and we know the K theory of $\mathbb{CP}^{n-1}$, this lets us easily compute the K theory of $\mathbb{RP}^{2n-1}$.
\end{ex}

\subsection{Localisation at $p$}
We wish to work with Krull-Schmidt categories everywhere, so we need to localise the K theory base change formalism to obtain $KU_p$ modules. This is a formal procedure, essentially given on integral objects by tensoring with the $p$-adic integers $\mathbb{Z}_p$ every place one sees a K group. For instance, to build the associated cohomology theory, we simply tensor with $\mathbb{Z}_p$. That this preserves the property of being a cohomology theory is immediate from flatness of $\mathbb{Z}_p$ over $\mathbb{Z}$, so this functor gives the associated spectrum $KU_p$ representing it.

We round out this section with a proof that $KU_p$ modules are a base change formalism on algebraic varieties.

\begin{prop}
    The base change formalism $Y\mapsto D^b_c(Y,KU_p)$ of sheaves of constructible $KU_p$ module spectra is a smoothly orientable Krull-Schmidt base change formalism on complex algebraic varieties, which satisfies condition (D).
\end{prop}

\begin{proof}
First, the fact that this is a base change formalism entails many compatibilities which follow from the construction, and Lurie's proper base change theorem (Chapter 7, \S 3 of \cite{LurieHTT}). One may find a streamlined proof of these properties in \cite{VolpeM}. For orientability, note that orientability is just an existence statement for elements in $K^{BM}_{2d}(X)\otimes_\mathbb{Z} \mathbb{Z}_p$. In particular, this is implied by the orientability of integral K-theory on complex manifolds, see Example \ref{3.9}. To check condition $(D)$, note that density holds trivially if the fibre bundle is trivial, and for any two points $x,y$ we may choose a contractible neighbourhood $U_{x,y}$ of $x$ and $y$. Restricting our geometric local system to $U_{x,y}$ gives a topologically trivial bundle, giving the density result.
It remains to check the Krull-Schmidt property of this base change formalism. We first claim it suffices to check the following conditions\footnote{This is slightly different to the conditions in \S \ref{sec-KS}, though the proof is the same.}.
\begin{itemize}
\item The ``ring of coefficients" $\mathbb{Z}_p[t,t^{-1}]:=\End(\mathbf{1}_\ast)$ is a graded complete local ring.
\item For any $\mathscr{F},\mathscr{G}$ in $D^b_c(X,KU_p)$, the group $\Hom^*(\mathscr{F},\mathscr{G})$ is a finitely generated graded $\mathbb{Z}_p[t,t^{-1}]$ module.
\item The category $D^b_c(X,KU_p)$ has split idempotents.
\end{itemize}
To see that these suffice, first note that the endomorphism algebra of any indecomposable object is a finite graded $\mathbb{Z}_p[t,t^{-1}]$ module by the second condition. The graded version of the idempotent lifting lemma (Corollary 7.5, \cite{eisenbud2013commutative}), and splitting of idempotents implies the endomorphism ring of any indecomposable object is local. The Krull-Schmidt property of functors then follows from the fact that for any finite morphism of local rings $\phi:R\rightarrow S$, we have $\phi(J(R))\subset J(S)$, which is immediate from Nakayama's Lemma.

It remains to check that these conditions hold for sheaves of $KU_p$ modules. The first condition is immediate by Bott Periodicity, as these are the homotopy groups of $KU_p$. For the finiteness of the second condition, since $\mathbb{Z}_p[t,t^{-1}]$ is Noetherian, one may apply open closed decomposition triangles to reduce to the case of morphisms between locally constant $KU_p$ modules on a smooth variety. We then can find a finite good cover of contractible open sets trivialising these $KU_p$ modules, and an induction with the Mayer-Vietoris sequence gives the result. Finally, to check idempotent completeness, we may assume our sheaves of $KU_p$ modules are constructible with respect to a fixed stratification $\lambda$. For a fixed stratification $\lambda$, we have the associated exit path $\infty$-category $EP_{\lambda,\infty}(X)$, and we may identify the category of $\lambda$ constructible sheaves of $KU_p$ modules with the functor category $[EP_{\lambda,\infty}(X),KU_p\perf]$ (see Theorem A.9.3 \cite{LurieHA}). As $KU_p\perf$ is accessible, this functor category is accessible (see \cite{LurieHTT} Proposition 5.4.4.3), and thus since this functor $\infty$ category is small, accessibility is equivalent to idempotent completeness (see Corollary 5.4.3.6 \cite{LurieHTT}).
\end{proof}

\subsection{Commuting diagrams}

In this section we prove the existence of the compatibility diagrams for the convolution isomorphism \ref{conviso}.

This can be broken into two distinct parts, the more formally $2$-categorical Proposition \ref{DIAGRAM}, and the orientation compatibility, Proposition \ref{convisocomp}.

Let us first recall the setup. Our space $Y$ is irreducible, with Zariski open set $U$, and $X,X'$ are two smooth spaces over $Y$. The following diagram will be our reference for the maps involved in the convolution isomorphism.

\[\begin{tikzcd}
    X_U\times_U X'_U\arrow[dd,"\tilde{g}_U"]\arrow[rr,"\tilde{f}_U"]\arrow[dr,"\hat{j}"]&&X'_U\arrow[dd,"g_U" near start]\arrow[dr,"\underline{j}'"]&\\
    &X\times_Y X'\arrow[dd,"\tilde{g}" near start]\arrow[rr,"\tilde{f}" near start]&&X'\arrow[dd,"g"]\\
    X_U\arrow[rr,"f_U" near start]\arrow[dr,"\underline{j}"']&&U\arrow[dr,"j"]&\\
    &X\arrow[rr,"f"]&&Y
\end{tikzcd}\]

Our first proposition is the following:

\begin{prop}\label{DIAGRAM}
The following diagram commutes, where the horizontal maps are our convolution isomorphisms, and the vertical maps are restriction followed by base change:
\[\begin{tikzcd}
    \Hom(f_!\_,g_*\_)\arrow[r,"\tau"]\arrow[d]&\Hom(\tilde{g}^*\_,\tilde{f}^!\_)\arrow[d]\\
    \Hom(j^*f_!\_,j^*g_*\_)\arrow[d,"\simeq"]&\Hom(\hat{j}^*\tilde{g}^*\_,\hat{j}^*\tilde{f}^!\_)\arrow[d,"\simeq"]\\
    \Hom({f_U}_!\underline{j}^*\_,{g_U}_*{\underline{j}'}^*\_)\arrow[r,"\tau"]&\Hom({\tilde{g}_U}^*\underline{j}^*\_,{\tilde{f}_U}^!{\underline{j}'}^*\_)
\end{tikzcd}\]
\end{prop}

Our orientation compatibility is the following:

\begin{prop}\label{convisocomp}
    There exist orientations of $\overline{\Delta_\alpha}$ and $X\times_Y X'$ such that the following diagram exists and is commutative, where the fundamental class of $\overline{\Delta_\alpha}$ restricted to $X_U\times_U X'_U$ maps to $\alpha_*\mathbf{1}$ under the associated convolution isomorphism.

    \[
      \begin{tikzpicture}[xscale=2,yscale=1]
        \node (H0) at (0,0) {$\Hom({f_U}_*\mathbf{1},{g_U}_*\mathbf{1})$};
        \node (H1) at (0,2) {$\Hom(f_*\mathbf{1},g_*\mathbf{1})$};
        \node (S0) at (2,0) {$S_{*,2d}(X_U\times_U X'_U)$};
        \node (S1) at (2,2) {$S_{*,2d}(X\times_Y X')$};
        \draw[->] (H0) to node[above] {$\sim$} (S0);
        \draw[->] (H1) to node[above] {$\sim$} (S1);
        \draw[->] (H1) to (H0);
        \draw[->] (S1) to (S0);
        \node (ur) at (3,2) {$[\overline{\Delta_\alpha}]$};
        \node (e1) at (2.72,2) {$\ni$};
        \node (lr) at (3,-1) {$[\Delta_\alpha]$};
        \node (e1) at (2.72,-.5) {\rotatebox{135}{$\in$}};
        \node (l) at (0,-1) {$\alpha_*\mathbf{1}$};
        \node (e1) at (0,-.6) {\rotatebox{90}{$\in$}};
        \draw[|->] (ur) to (lr);
        \draw[|->] (l) to (lr);
      \end{tikzpicture}\]
\end{prop}

For notational convenience, in the proof of this proposition we will
use $(\_,\_)$ to denote morphism sets, and we will only use $f,g$ and
$j$, noting that the decorations are uniquely determined by the location
within the diagram. 

\begin{proof}
We first prove Proposition \ref{DIAGRAM}. We may expand the diagram in Proposition \ref{DIAGRAM} into the following:

\[\begin{tikzcd}
    (f_!,g_*)\arrow[r]\arrow[d]&(\_,f^!g_*)\arrow[r]\arrow[d]&(\_,{g}_*{f}^!)\arrow[d]\arrow[r]&({g}^*,{f}^!)\arrow[d]\\
    (j^*f_!,j^*g_*)\arrow[d]\arrow[r]&(\_,f^!j_*j^*g_*)\arrow[d]&(\_,{g}_*{j}_*{j}^*{f}^!)\arrow[r]\arrow[d]&({j}^*{g}^*,{j}^*{f}^!)\arrow[d]\\
    ({f}_! {j}^*,{g}_* {{j}'}^*)\arrow[r]&(\_,{{j}}_*
    {f}^!{g}_* {{j}}^*)\arrow[r]&
    (\_ ,{{j}}_* {{g}}_*
    {{f}}^!{j}^*)\arrow[r]&({{g}}^*{j}^*,{{f}}^!{j}^*)
\end{tikzcd}\]
Only the commutativity of the middle square is not standard, and its commutativity follows from the commutativity of the following diagram:
\[\begin{tikzcd}
{(\_,f^!g_*)} \arrow[rr] \arrow[d]                 &                                       & {(\_,g_*f^!)} \arrow[d] \arrow[r]        & {(\_,g_*j_*j^*f^!)} \arrow[dd]             &                   \\
{(\_,f^!g_!)} \arrow[rr] \arrow[rd] \arrow[d]      &                                       & {(\_,g_!f^!)} \arrow[d]                  &                                          &                   \\
{(\_,f^!j_*j^*g_!)} \arrow[r] \arrow[d]            & {(\_,j_*j^*f^!g_!)} \arrow[r] \arrow[d] & {(\_,j_*j^*g_!f^!)} \arrow[d] \arrow[r]  & {(\_,j_*g_*j^*f^!)} \arrow[rddd] \arrow[d] &                   \\
{(\_,j_*f^!j^*g_!)} \arrow[ru] \arrow[r] \arrow[d] & {(j^!,j^!f^!g_!)} \arrow[d] \arrow[r] & {(j^!,j^!g_!f^!)} \arrow[d] \arrow[r]  & {(j^*,g_*j^!f^!)} \arrow[d]              &                   \\
{(\_,j_*f^!g_!j^*)} \arrow[d] \arrow[rd]           & {(j^!,f^!j^!g_!)} \arrow[d]           & {(j^!,g_!j^!f^!)} \arrow[d] \arrow[r]  & {(j^*,g_!j^!f^!)} \arrow[d]              &                   \\
{(\_,j_*f^!g_*j^*)} \arrow[r]                      & {(j^!,f^!g_!j^!)} \arrow[r]           & {(j^!,g_!f^!j^!)} \arrow[r] \arrow[rd] & {(j^*,g_!f^!j^!)} \arrow[d]              & {(\_,j_*g_*f^!j^*)} \\
                                                 &                                       &                                        & {(\_,j_*g_!f^!j^!)} \arrow[ru]             &                  
\end{tikzcd}\]
The commutativity of the internal faces are all standard compatibilities of base change and naturality.
\end{proof}

\begin{remark}
  A general statement proving commutativity of diagrams of this type is to be found in the thesis of the first author \cite{Hone-thesis}.
\end{remark}

It remains to prove Proposition \ref{convisocomp}. This entails proving the commutativity of the diagram, and for compatible choices of orientation, that the restriction of $[\overline{\Delta_\alpha}]$ to $S_{2d_X}^!(X_U\times_U X'_U)$ corresponds to $\alpha_*\mathbf{1}$ under the convolution isomorphism.

To show the existence of this diagram, we evaluate on the constant sheaf, and use the chosen orientation $\mathbf{1}\xrightarrow{\gamma}\omega_X[-2d]$ of $X$
to give the identifications with $S_{*,2d}(X\times_Y X')$.\\
\begin{adjustwidth}{-2.2cm}{0pt}
\begin{tikzcd}[
  column sep={10em,between origins}]
{\Hom(f_*\mathbf{1},g_*\mathbf{1})} \arrow[r] \arrow[d] & {\Hom(f_!\mathbf{1},g_*\mathbf{1})} \arrow[r] \arrow[d] & {\Hom(\tilde{g}^*\mathbf{1},\tilde{f}^!\mathbf{1})} \arrow[d] \arrow[r,"\gamma"] & {\Hom(\mathbf{1},\tilde{f}^!\omega[-2d])} \arrow[d] \arrow[r] & {S_{*,2d}(X\times_Y X')} \arrow[d] \\
{\Hom({f_U}_*\mathbf{1},{g_U}_*\mathbf{1})} \arrow[r]   & {\Hom({f_U}_!\mathbf{1},{g_U}_*\mathbf{1})} \arrow[r]   & {\Hom(\tilde{g}_U^*\mathbf{1},\tilde{f}_U^!\mathbf{1})} \arrow[r,"\gamma_U"]       & {\Hom(\mathbf{1},{\tilde{f}_U}^!\omega[-2d])} \arrow[r]       & {S_{*,2d}(X_U\times_U X'_U)}      
\end{tikzcd}
\end{adjustwidth}
Here the commutativity of the second square is Proposition \ref{DIAGRAM}. Thus it remains to prove the identification of the fundamental class along this isomorphism.
\begin{prop}
Given an isomorphism $\alpha$ over $U$, such that $X_U,X_U'$ are smooth:
    \[\begin{tikzcd}
    X_U\arrow[rr,"\alpha"]\arrow[dr,"f"']&&X_U'\arrow[dl,"g"]\\
        & U
    \end{tikzcd}\]
    Then for a choice of orientation $\gamma$ of $X_U$, and induced compatible orientation of $\Delta_\alpha$, the element $[\Delta_\alpha]$ in $H_{2d}(X_U\times_U X_U')$ corresponds to $\alpha_*\mathbf{1}$ in $\Hom(f_*\mathbf{1},g_*\mathbf{1})$ via the convolution isomorphism induced by $\gamma$. We may also assume that this orientation of $\Delta_\alpha$ arises as restriction from an orientation of $\overline{\Delta_\alpha}$.
\end{prop}

\begin{proof}
We first construct the orientations required by choosing an orientation of $\overline{\Delta_\alpha}$ via resolution of singularities. We may then restrict this to $\Delta_\alpha$, and transport structure to $X_U$ to get our desired compatible orientations. From here all convolution morphisms and fundamental classes are with respect to this choice.

First, let's show that this fundamental class morphism $[\Delta_\alpha]:\mathbf{1}\rightarrow \omega_{X_U\times_U X'_U}[-2d]$ is isomorphic (after whiskering) to an evaluation on the constant sheaf of a morphism of functors.
That is, there are canonical morphisms of functors \[\tilde{g}^*\rightarrow\Delta_*\rightarrow \Delta_!\rightarrow \tilde{f}^!\] such that the following diagram commutes, and the composite is the fundamental class in $H_{2d}(X_U\times_U X'_U)$:

    \[\begin{tikzcd}
        \mathbf{1}\arrow[r]\arrow[dr]&\tilde{g}^*\mathbf{1}\arrow[r]\arrow[d]&\Delta_!\mathbf{1}\arrow[r]&\tilde{f}^!\mathbf{1}\arrow[dr]\\
        &\Delta_*\mathbf{1}\arrow[ur]\arrow[r]&\Delta_!\omega_{\Delta}[-2d]\arrow[u]\arrow[r]&\tilde{f}^!\omega_{\Delta}[-2d]\arrow[u]\arrow[r]&\omega_{X_1\times X_2}[-2d]
    \end{tikzcd}\]
The commutativity of the rightmost triangle follows from the fact that
our orientations were chosen compatibly.

This reduces the problem to a coherence problem for pseudofunctors, so consider the following diagram.   \[\begin{tikzcd}     \text{Id}\arrow[r]\arrow[d,equal]&\tilde{g}_*{\Delta_{\alpha}}_*\arrow[r]&\tilde{g}_*{\Delta_{\alpha}}_!\arrow[r]\arrow[d]&\tilde{g}_*\tilde{f}^!\arrow[d]\\
\text{Id}\arrow[d]\arrow[rr]&&\tilde{g}_!{\Delta_{\alpha}}_!\arrow[d]\arrow[r]&\tilde{g}_!\tilde{f}^!\arrow[d]\\
    f^!f_!\arrow[d]\arrow[rr]&&f^!f_!\tilde{g}_!\Delta_{{\alpha}!}\arrow[d]\arrow[r]&f^!f_!\tilde{g}_!\tilde{f}^!\arrow[d]\\
    f^!g_!\arrow[d,equal]\arrow[dr]&&f^!f_!\tilde{g}_!\Delta_!\Delta^!\tilde{f}^!\arrow[dl]\arrow[d]\arrow[ur]&f^!g_!\tilde{f}_!\tilde{f}^!\arrow[d]\\
    f^!g_!\arrow[r]\arrow[rrr,equal,bend right]&f^!g_!\tilde{g}_!\Delta_!\Delta^!\tilde{f}^!\arrow[r]&f^!g_!\tilde{f}_!\Delta_!\Delta^!\tilde{f}^!\arrow[r]\arrow[ur]&f^!g_!\arrow[r]&f^!g_*
    \end{tikzcd}\]
Going clockwise around the diagram gives the mate of $[\Delta_\alpha]$, by the previous discussion, while going anticlockwise yields the mate of the morphism $f_!\xrightarrow{\alpha_!} g_!\rightarrow g_*$ which equals the composite $f_!\rightarrow f_*\xrightarrow{\alpha_*}g_*$, giving the desired compatibility.   

All squares in this diagram commute by naturality, or using that $S_!$ is a pseudofunctor. Only the curved identity morphism is not immediate, this is $f^!g_!$ applied to the following diagram:\[\begin{tikzcd}   \Id\arrow[r]\arrow[dr]\arrow[ddrr,bend right,equal]&\tilde{g}_!\Delta_!\Delta^!\tilde{f}^!\arrow[r]&\tilde{f}_!\Delta_!\Delta^!\tilde{f}^!\arrow[d]\\
    &(\tilde{f}\circ \Delta)_!(\tilde{f}\circ \Delta)^!\arrow[dr]\arrow[ur]&\tilde{f}_!\tilde{f}^!\arrow[d]\\
    &&\Id
    \end{tikzcd}\]
This then commutes by the definition of the horizontal isomorphisms, and the general unit compatibilities of pseudofunctors. 
\end{proof}

\bibliography{Biblio}

\begin{thebibliography}{10}

\bibitem{Achar}
Pramod~N. Achar.
\newblock {\em Perverse sheaves and applications to representation theory},
  volume 258 of {\em Mathematical Surveys and Monographs}.
\newblock American Mathematical Society, Providence, RI, [2021] \copyright
  2021.

\bibitem{AMRW}
Pramod~N. Achar, Shotaro Makisumi, Simon Riche, and Geordie Williamson.
\newblock Koszul duality for {K}ac-{M}oody groups and characters of tilting
  modules.
\newblock {\em J. Amer. Math. Soc.}, 32(1):261--310, 2019.

\bibitem{ARKoszul}
Pramod~N. Achar and Simon Riche.
\newblock Modular perverse sheaves on flag varieties, {II}: {K}oszul duality
  and formality.
\newblock {\em Duke Math. J.}, 165(1):161--215, 2016.

\bibitem{ARloop}
Pramod~N. Achar and Simon Riche.
\newblock Reductive groups, the loop {G}rassmannian, and the {S}pringer
  resolution.
\newblock {\em Invent. Math.}, 214(1):289--436, 2018.

\bibitem{AB-Clifford}
M.~F. Atiyah, R.~Bott, and A.~Shapiro.
\newblock Clifford modules.
\newblock {\em Topology}, 3(suppl):3--38, 1964.

\bibitem{BL}
Joseph Bernstein and Valery Lunts.
\newblock {\em Equivariant sheaves and functors}, volume 1578 of {\em Lecture
  Notes in Mathematics}.
\newblock Springer-Verlag, Berlin, 1994.

\bibitem{BBD}
A.~A. Be\u{\i}linson, J.~Bernstein, and P.~Deligne.
\newblock Faisceaux pervers.
\newblock In {\em Analysis and topology on singular spaces, {I} ({L}uminy,
  1981)}, volume 100 of {\em Ast\'{e}risque}, pages 5--171. Soc. Math. France,
  Paris, 1982.

\bibitem{BorhoMacPherson}
Walter Borho and Robert MacPherson.
\newblock Partial resolutions of nilpotent varieties.
\newblock In {\em Analysis and topology on singular spaces, {II}, {III}
  ({L}uminy, 1981)}, volume 101-102 of {\em Ast\'{e}risque}, pages 23--74. Soc.
  Math. France, Paris, 1983.

\bibitem{bott1959stable}
Raoul Bott.
\newblock The stable homotopy of the classical groups.
\newblock {\em Annals of Mathematics}, 70(2):313--337, 1959.

\bibitem{CMP}
James Carlson, Stefan M\"{u}ller-Stach, and Chris Peters.
\newblock {\em Period mappings and period domains}, volume~85 of {\em Cambridge
  Studies in Advanced Mathematics}.
\newblock Cambridge University Press, Cambridge, 2003.

\bibitem{CD}
Denis-Charles Cisinski and Fr\'{e}d\'{e}ric D\'{e}glise.
\newblock {\em Triangulated categories of mixed motives}.
\newblock Springer Monographs in Mathematics. Springer, Cham, [2019] \copyright
  2019.

\bibitem{dCM}
Mark Andrea~A. de~Cataldo and Luca Migliorini.
\newblock The decomposition theorem, perverse sheaves and the topology of
  algebraic maps.
\newblock {\em Bull. Amer. Math. Soc. (N.S.)}, 46(4):535--633, 2009.

\bibitem{deJong}
A.~J. de~Jong.
\newblock Smoothness, semi-stability and alterations.
\newblock {\em Inst. Hautes \'{E}tudes Sci. Publ. Math.}, (83):51--93, 1996.

\bibitem{DeligneHodgeIII}
Pierre Deligne.
\newblock Th\'{e}orie de {H}odge. {II}.
\newblock {\em Inst. Hautes \'{E}tudes Sci. Publ. Math.}, (40):5--57, 1971.

\bibitem{Eberhardt2}
Jens~Niklas Eberhardt.
\newblock Springer motives.
\newblock {\em Proc. Amer. Math. Soc.}, 149(5):1845--1856, 2021.

\bibitem{Eberhardt1}
Jens~Niklas Eberhardt.
\newblock {$K$}-motives and {K}oszul duality.
\newblock {\em Bull. Lond. Math. Soc.}, 54(6):2232--2253, 2022.

\bibitem{eisenbud2013commutative}
David Eisenbud.
\newblock {\em Commutative algebra: with a view toward algebraic geometry},
  volume 150.
\newblock Springer Science \& Business Media, 2013.

\bibitem{EW-Satake}
Ben Elias and Geordie Williamson.
\newblock Quantum satake via {K}-theory soergel bimodules, 2023.
\newblock In preparation.

\bibitem{FWpsmooth}
Peter Fiebig and Geordie Williamson.
\newblock Parity sheaves, moment graphs and the {$p$}-smooth locus of
  {S}chubert varieties.
\newblock {\em Ann. Inst. Fourier (Grenoble)}, 64(2):489--536, 2014.

\bibitem{GR1}
Dennis Gaitsgory and Nick Rozenblyum.
\newblock {\em A study in derived algebraic geometry. {V}ol. {I}.
  {C}orrespondences and duality}, volume 221 of {\em Mathematical Surveys and
  Monographs}.
\newblock American Mathematical Society, Providence, RI, 2017.

\bibitem{hatcher}
Allen Hatcher.
\newblock {\em Algebraic topology}.
\newblock Cambridge University Press, Cambridge, 2002.

\bibitem{hatcherVB}
Allen Hatcher.
\newblock Vector bundles and {K}-theory.
\newblock {\em See http://www. math. cornell. edu/\~{} hatcher}, 2003.

\bibitem{Hirzebruch}
F.~Hirzebruch.
\newblock {\em Topological methods in algebraic geometry.}
\newblock Springer-Verlag New York, Inc., New York,, enlarged edition, 1966.

\bibitem{Hone-thesis}
Chris Hone.
\newblock Geometric extensions in representation theory, 2023.
\newblock PhD thesis, in preparation.

\bibitem{JMW-survey}
Daniel Juteau, Carl Mautner, and Geordie Williamson.
\newblock Perverse sheaves and modular representation theory.
\newblock In {\em Geometric methods in representation theory. {II}}, volume
  24-II of {\em S\'{e}min. Congr.}, pages 315--352. Soc. Math. France, Paris,
  2012.

\bibitem{JMW}
Daniel Juteau, Carl Mautner, and Geordie Williamson.
\newblock Parity sheaves.
\newblock {\em J. Amer. Math. Soc.}, 27(4):1169--1212, 2014.

\bibitem{JWKumar}
Daniel Juteau and Geordie Williamson.
\newblock Kumar's criterion modulo {$p$}.
\newblock {\em Duke Math. J.}, 163(14):2617--2638, 2014.

\bibitem{kashiwara2002sheaves}
M.~Kashiwara and P.~Schapira.
\newblock {\em Sheaves on Manifolds}.
\newblock Grundlehren der mathematischen Wissenschaften. Springer Berlin
  Heidelberg, 2002.

\bibitem{KL}
David Kazhdan and George Lusztig.
\newblock Proof of the {D}eligne-{L}anglands conjecture for {H}ecke algebras.
\newblock {\em Invent. Math.}, 87(1):153--215, 1987.

\bibitem{Kodaira}
K.~Kodaira.
\newblock On compact analytic surfaces. {II}, {III}.
\newblock {\em Ann. of Math. (2) 77 (1963), 563--626; ibid.}, pages 1--40,
  1963.

\bibitem{Krause}
Henning Krause.
\newblock Krull-{S}chmidt categories and projective covers.
\newblock {\em Expo. Math.}, 33(4):535--549, 2015.

\bibitem{liuzheng}
Yifeng Liu and Weizhe Zheng.
\newblock Enhanced six operations and base change theorem for higher {A}rtin
  stacks.
\newblock {\em arXiv preprint arXiv:1211.5948}, 2012.

\bibitem{LurieHTT}
Jacob Lurie.
\newblock {\em Higher topos theory}, volume 170 of {\em Annals of Mathematics
  Studies}.
\newblock Princeton University Press, Princeton, NJ, 2009.

\bibitem{LurieHA}
Jacob Lurie.
\newblock Higher algebra, 2017.

\bibitem{LurieSAG}
Jacob Lurie.
\newblock Spectral algebraic geometry.
\newblock {\em preprint}, 2018.

\bibitem{LW}
George Lusztig and Geordie Williamson.
\newblock Billiards and tilting characters for {$\rm SL_3$}.
\newblock {\em SIGMA Symmetry Integrability Geom. Methods Appl.}, 14:Paper No.
  015, 22, 2018.

\bibitem{mann}
Lucas Mann.
\newblock A $p$-adic 6-functor formalism in rigid-analytic geometry.
\newblock {\em arXiv preprint arXiv:2206.02022}, 2022.

\bibitem{PM}
Peter~J. McNamara.
\newblock Non-perverse parity sheaves on the flag variety.
\newblock {\em Nagoya Math. J.}, 249:1--10, 2023.

\bibitem{MV}
Luca Migliorini and Vivek Shende.
\newblock A support theorem for {H}ilbert schemes of planar curves.
\newblock {\em J. Eur. Math. Soc. (JEMS)}, 15(6):2353--2367, 2013.

\bibitem{Nakajima}
Hiraku Nakajima.
\newblock Quiver varieties and finite-dimensional representations of quantum
  affine algebras.
\newblock {\em J. Amer. Math. Soc.}, 14(1):145--238, 2001.

\bibitem{ngo}
Bao~Ch\^{a}u Ng\^{o}.
\newblock Le lemme fondamental pour les alg\`{e}bres de {L}ie.
\newblock {\em Publ. Math. Inst. Hautes \'{E}tudes Sci.}, (111):1--169, 2010.

\bibitem{IKT}
Tudor P{\u{a}}durariu.
\newblock Intersection {K}-theory.
\newblock {\em arXiv preprint arXiv:2103.06223}, 2021.

\bibitem{DanP}
Dan Petersen.
\newblock What (if anything) unifies stable homotopy theory and
  {G}rothendieck’s six functors formalism?
\newblock
  {https://mathoverflow.net/questions/170319/what-if-anything-unifies-stable-homotopy-theory-and-grothendiecks-six-functor},
  2014.

\bibitem{RSW}
Simon Riche, Wolfgang Soergel, and Geordie Williamson.
\newblock Modular {K}oszul duality.
\newblock {\em Compos. Math.}, 150(2):273--332, 2014.

\bibitem{rudyak}
Yuli~B. Rudyak.
\newblock {\em On {T}hom spectra, orientability, and cobordism}.
\newblock Springer Monographs in Mathematics. Springer-Verlag, Berlin, 1998.
\newblock With a foreword by Haynes Miller.

\bibitem{Sa}
Morihiko Saito.
\newblock A young person's guide to mixed {H}odge modules.
\newblock In {\em Hodge theory and {$L^2$}-analysis}, volume~39 of {\em Adv.
  Lect. Math. (ALM)}, pages 517--553. Int. Press, Somerville, MA, 2017.

\bibitem{VolpeM}
Marco Volpe.
\newblock The six operations in topology.
\newblock {\em arXiv preprint arXiv:2110.10212}, 2021.

\bibitem{W}
Geordie Williamson.
\newblock The {H}odge theory of the decomposition theorem.
\newblock Number 390, pages Exp. No. 1115, 335--367. 2017.

\bibitem{Wtors}
Geordie Williamson.
\newblock On torsion in the intersection cohomology of {S}chubert varieties.
\newblock {\em J. Algebra}, 475:207--228, 2017.

\bibitem{W-explosion}
Geordie Williamson.
\newblock Schubert calculus and torsion explosion.
\newblock {\em J. Amer. Math. Soc.}, 30(4):1023--1046, 2017.
\newblock With a joint appendix with Alex Kontorovich and Peter J. McNamara.

\bibitem{W-ICM}
Geordie Williamson.
\newblock Parity sheaves and the {H}ecke category.
\newblock In {\em Proceedings of the {I}nternational {C}ongress of
  {M}athematicians---{R}io de {J}aneiro 2018. {V}ol. {I}. {P}lenary lectures},
  pages 979--1015. World Sci. Publ., Hackensack, NJ, 2018.

\end{thebibliography}
\bibliographystyle{plain}

\end{document}